\documentclass[12pt]{amsart}
\usepackage[T1]{fontenc}
\usepackage{amscd, amsfonts, amsmath, amssymb, amsthm, enumerate, flexisym, float, graphicx, hyperref, pdfpages,tikz, mathtools,comment, pinlabel, subcaption, tabu, tikz, tikz-cd, url, wasysym, soul}
\usepackage[autostyle]{csquotes}
\usepackage[all,cmtip]{xy}
\hypersetup{colorlinks=true, linkcolor=blue,citecolor=blue}
\usetikzlibrary{cd, arrows, calc, knots, patterns, positioning, shapes.callouts, shapes.geometric, shapes.symbols}
\tikzset{>=stealth', arrow/.style={->}}
\makeatletter
\@namedef{subjclassname@2020}{Mathematics Subject Classification \textup{2020}}

\newsavebox{\@brx}
\newcommand{\llangle}[1][]{\savebox{\@brx}{$\m@th{#1\langle}$}%
	\mathopen{\copy\@brx\kern-0.5\wd\@brx\usebox{\@brx}}}
\newcommand{\rrangle}[1][]{\savebox{\@brx}{$\m@th{#1\rangle}$}%
	\mathclose{\copy\@brx\kern-0.5\wd\@brx\usebox{\@brx}}}

\makeatother

\newtheorem{theorem}{Theorem}[section]
\newtheorem{cor}[theorem]{Corollary}
\newtheorem{prop}[theorem]{Proposition}
\newtheorem{lem}[theorem]{Lemma}
\newtheorem{rem}[theorem]{Remark}

\newtheorem{example}[theorem]{Example}
\newtheorem{prob}[theorem]{Problem}
\newtheorem{conj}[theorem]{Conjecture}

\newcommand{\Map}{\mathrm{Map}}

\newcommand{\Z}{\mathbb{Z}}

\newcommand{\R}{\mathbb{R}}

\newcommand{\M}{\mathcal{M}}
\newcommand{\D}{\mathcal{D}}
\newcommand{\C}{\mathcal{C}}
\newcommand{\E}{\mathcal{E}}
\newcommand{\W}{\mathcal{W}}

\newcommand{\Ca}{\textbf{C}}
\newcommand{\ca}{\textbf{c}}
\newcommand{\Za}{\textbf{Z}}
\newcommand{\Ba}{\textbf{B}}
\newcommand{\mcg}{\mathcal{M}}
\newcommand{\Ha}{\textbf{H}}

\newcommand{\stab}{\operatorname{Stab}}
\newcommand{\Aut}{\operatorname{Aut}}
\newcommand{\Adj}{\operatorname{Adj}}
\newcommand{\Conj}{\operatorname{Conj}}
\newcommand{\Inn}{\operatorname{Inn}}
\everymath{\displaystyle}

\title{Dehn quandles of surfaces and their bounded cohomology}
\author{Pankaj Kapari}
\address{Department of Mathematical Sciences, Indian Institute of Science Education and Research (IISER) Mohali,  Sector 81, S. A. S. Nagar, P. O. Manauli,    Punjab 140306, India.}
\email{pankajkapri02@gmail.com}

\author{Deepanshi Saraf}
\address{Department of Mathematics, Indian Institute of Science Education and Research (IISER) Tirupati, Srinivasapuram, Yerpedu, Tirupati, Andhra Pradesh, 517619, India.}
\email{saraf.deepanshi@gmail.com}

\author{Mahender Singh}
\address{Department of Mathematical Sciences, Indian Institute of Science Education and Research (IISER) Mohali,  Sector 81, S. A. S. Nagar, P. O. Manauli,    Punjab 140306, India.}
\email{mahender@iisermohali.ac.in}


\begin{document}

\begin{abstract}
We introduce new families of quandles associated with closed orientable surfaces that can distinguish surfaces up to homeomorphism. These quandles generalize the classical Dehn quandle and are defined, respectively, on isotopy classes of unoriented  essential  closed curves and on integral weighted multicurves. We establish their fundamental algebraic properties and construct a natural quandle covering relating the two families. We further investigate their large-scale geometry and prove that they are unbounded with respect to the quandle metric. We determine their second bounded quandle cohomology and show that its dimension equals the cardinality of the continuum. We also establish a quandle-theoretic analogue of the Gromov Mapping Theorem, showing that the natural projection from an abelian quandle extension onto its base quandle induces an injection on bounded quandle cohomology in every degree. Finally, motivated by recent developments in quandle algebras, we study idempotents in the integral quandle algebras associated with the classical Dehn quandle. In particular, we prove that the integral quandle algebra of the Dehn quandle of the torus admits no non-trivial idempotents.
\end{abstract}

\maketitle

\section{Introduction}
Quandles are algebraic structures whose defining operation encodes the Reidemeister moves of link diagrams.  They arise naturally in diverse areas of mathematics, including knot theory, group theory, surface theory, quantum algebra, symmetric spaces, Hopf algebras, and the study of the quantum Yang--Baxter equation. A (co)homology theory for quandles was first introduced by Fenn, Rourke, and Sanderson \cite{MR1364012} via a homotopy-theoretic approach using classifying spaces. Subsequently, Carter et al. \cite{MR1990571} applied quandle cohomology to knot theory through the construction of state-sum invariants. These theories were later generalized in \cite{MR1994219, MR3558231}. More recently, Szymik \cite{MR3937311} interpreted quandle cohomology within the framework of Quillen cohomology.
\par

Motivated by bounded cohomology of groups, as introduced by Johnson \cite{MR0374934} and subsequently developed by Gromov \cite{MR0686042}, K\k{e}dra \cite{MR4779104} recently introduced a bounded cohomology theory for quandles and related it to a natural metric based on their inner symmetries. He showed that a quandle with finitely many connected components is bounded with respect to this metric if and only if the comparison map from its second bounded quandle cohomology to its ordinary second quandle cohomology is injective. Furthermore, he proved that fundamental quandles of non-trivial knots in the 3-sphere and free quandles are unbounded, implying the non-triviality of their second bounded cohomology. Consequently, the second bounded cohomology of the fundamental quandle of a knot detects the unknot. More recently, it was shown in \cite{arXiv:2502.04069} that,  for every link other than the unknot and the Hopf link, the second bounded cohomology of its fundamental quandle has dimension equal to the cardinality of the continuum. Furthermore, for every $n \ge 3$, it was proved that the second bounded cohomology of the fundamental $n$-quandle of a hyperbolic knot other than the figure-eight knot has dimension equal to the cardinality of the continuum. 
\par

This paper is motivated by quandles arising from surfaces and by recent developments in bounded cohomology. Let $S_g$ be a closed orientable surface of genus $g \ge 1$, and let $\D_g$ be the set of isotopy classes of unoriented essential simple closed curves on $S_g$. The binary operation $$\alpha * \beta= T_\beta(\alpha),$$ where $\alpha, \beta \in \D_g$ and $T_\beta$ denotes the Dehn twist along $\beta$, equips $\D_g$ with the structure of a quandle, known as the Dehn quandle of the surface $S_g$. These quandles were introduced by Zablow \cite{MR1967241, MR2699808}. Notably, Niebrzydowski and Przytycki \cite{MR2583322} proved that the Dehn quandle of the torus is isomorphic to the fundamental quandle of the trefoil knot. However, this is an exceptional phenomenon that does not persist in general \cite{MR3197056}. In \cite{MR4642200}, two approaches were developed to obtain explicit presentations for such quandles.
\par

In this paper, we introduce new families of quandles that generalise the classical Dehn quandles and arise from collections of unoriented essential closed curves on closed orientable surfaces, as well as from integral weighted multicurves. The well-known Goldman Lie algebra \cite{MR0846929} of such a surface is generated by oriented isotopy classes of closed curves. Our motivation for defining quandle structures on such collections of curves also stems from exploring a possible connection between integral quandle algebras and the Goldman Lie algebra. It is known that the quandle structure on the classical Dehn quandle extends to the set of measured foliations on the torus \cite{MR3205562}. For surfaces of higher genus, multicurves with positive real weights are dense in the space of measured foliations \cite{MR0956596}. This motivates us to define a quandle structure on integral weighted multicurves and explore whether this structure can be extended to the entire space of measured foliations. Remarkably, these quandles classify closed orientable surfaces of genus at least three. We establish their fundamental algebraic properties and compute the dimension of their second bounded quandle cohomology. Along the way, we prove an analogue of the Gromov Mapping Theorem for abelian quandle extensions. Finally, we study idempotents in the integral quandle algebras associated with Dehn quandles and give a precise description of idempotents in the case of the torus.
\par

The paper is organized as follows. Section \ref{section preliminary} reviews the basic notions of quandles, bounded group cohomology, and bounded quandle cohomology. In Section \ref{section quandles on surfaces}, we introduce two new families of quandles associated with the closed orientable surface $S_g$ of genus $g\geq 1$, both containing the classical Dehn quandle $\D_g$ as a subquandle. The first, denoted by $\C_g$, is defined on the set of isotopy classes of unoriented essential closed curves on $S_g$, while the second, denoted by $\W_g$, is defined on the set of integral weighted multicurves (Propositions \ref{Cg quandle} and \ref{Dg quandle}). We establish their basic algebraic properties and construct a natural quandle covering $\C_g\to\W_g^+$, where $\W_g^+$ is the subquandle of positively weighted multicurves (Theorem \ref{theorem covering}). Using a description of their inner automorphism groups, we show that these quandles distinguish closed orientable surfaces of genus greater than two (Proposition \ref{classifying surfaces}).
\par
Section \ref{section Metrics on Dehn quandles} investigates the geometry of these quandles. We prove that each is unbounded with respect to the quandle metric (Proposition \ref{unbounded quandles}). Further, for the classical Dehn quandle $\D_g$ with $g\geq2$, we show that the quandle metric and the curve complex metric are not comparable (Proposition \ref{prop:comparison_of_distance}). In Section \ref{section Bounded cohomology of Dehn quandles}, we study the bounded cohomology of the quandles introduced above. We prove that the second bounded quandle cohomology of each has dimension equal to the cardinality of the continuum (Theorem \ref{thm:D_g_cohomolofy_id}). We also establish a quandle-theoretic analogue of the Gromov Mapping Theorem, showing that for an abelian extension $E$ of a quandle $X$, the natural projection $E\to X$ induces an injection on bounded quandle cohomology in every degree (Theorem \ref{gmt for quandles}).
\par
Finally, Section \ref{section Idempotents in integral quandle algebras} concerns integral quandle algebras, non-associative algebras that provide a representation-theoretic framework for studying quandles. We analyze idempotents in the integral quandle algebras associated with the Dehn quandles (Proposition \ref{idempotent structure}). In particular, we prove that the integral quandle algebra of the Dehn quandle of the torus admits no non-trivial idempotents (Theorem \ref{idempotents in dehn quandle of torus}).
\medskip

\section{Preliminaries}\label{section preliminary}
This section reviews the essential preliminaries used throughout the paper.

\subsection{Quandles} To establish our conventions, recall that a {\it quandle} is a set $X$ equipped with a binary operation $\ast$ that satisfies the following axioms:
	\begin{enumerate}
		\item $x \ast x = x$ for all $x \in X$.
		\item For each pair $x, y \in X$, there exists a unique $z \in X$ such that $x = z \ast y$.
		\item $(x \ast y) \ast z = (x \ast z) \ast (y \ast z)$ for all $x, y, z \in X$.
	\end{enumerate}

An algebraic structure satisfying the last two axioms is called a {\it rack}. The second quandle axiom is equivalent to the existence of a dual binary operation on $X$, denoted by $(x, y) \mapsto x \ast^{-1} y$, satisfying $x \ast y = z$ if and only if $x = z \ast^{-1} y$ for all $x, y, z \in X$. The morphisms of quandles are defined in the usual way. For each $x \in X$, the map $S_x: X \rightarrow X$ given by $S_x(y) = y \ast x$ is an automorphism of $X$ that fixes $x$. The group $\Inn(X)$ generated by such automorphisms is called the {\it inner automorphism group} of $X$. There is a natural left action of $\Inn(X)$ on $X$, defined by $(f, x) \mapsto f(x)$ for $f \in \Inn(X)$ and $x \in X$. The orbits of this action are called the {\it connected components} of $X$, and the quandle is called {\it connected} if it has only one orbit.
\par

Knots and links constitute a rich source of quandle structures.

\begin{example}\label{lem:cosets_quandle}
{\rm Let $L$ be an oriented link in $\mathbb{S}^3$. In \cite[Section 4.5]{MR2628474} and \cite[Section 6]{MR0672410}, Joyce and Matveev independently gave a topological construction of the {\it fundamental quandle} $Q(L)$ of $L$, and proved it to be an invariant of the isotopy type of $L$. Further, they proved that $Q(L)$ can also be obtained from a regular diagram $D$ of $L$. Suppose that $D$ has $s$ arcs and $t$ crossings. If we assign labels $x_{1}, \ldots, x_{s}$ to the arcs of $D$ and introduce a relation $r_l$ given by  $x_k*x_j=x_i$ or $x_k*^{-1}x_j=x_i$ at the $l$-th crossing of $D$ depending on the sign of this crossing, then $$Q(L) \cong \langle x_1, \ldots, x_s \mid r_1,  \dots, r_t\rangle.$$}
\end{example}

Beyond knots and surfaces, quandle structures arise naturally in a wide range of settings. In particular, groups provide an important source of examples. Notably, each group $G$ can be turned into a quandle $\Conj(G)$ by defining $x*y=yxy^{-1}$. We will rely heavily on the following group-theoretic construction of a quandle. 

\begin{example}\label{lem:cosets_quandle 2}
{\rm Let $G$ be a group, $\{z_i \mid i \in I \}$ a set of elements of $G$, and $\{H_i \mid i \in I \}$ a set of subgroups of $G$ such that $H_i \le C_G(z_i)$ for each $i$, where $C_G(z_i)$ is the centraliser of $z_i$ in $G$. Then we can define a quandle structure on the disjoint union $\sqcup_{i \in I} G/H_i$ of right cosets by
			$$H_i x \ast H_j y = H_i z_ixy^{-1}z_j^{-1} y.$$ 
We denote this quandle by $\sqcup_{i \in I} (G/H_i, z_i )$. }
\end{example}

It is known that any quandle $X$ can be expressed as $\sqcup_{i\in I}(G/H_i,z_i)$ for a suitable group $G$ acting on the set $X$ \cite[Section 2.4]{MR2628474}. In fact, Joyce \cite{MR2628474} considered the case  $G=\Aut(X)$, while Nosaka \cite[Chapter 2]{MR3729413} took $G=\Adj(X)$ and noted that we may also take $G=\Inn(X)$. Since this result will be required later, we provide a proof below.

\begin{prop}
\label{prop:quandle_as_coset_quandle}
Let $X$ be a quandle. Then $X$ is isomorphic to $\sqcup_{i\in I}(\Inn(X)/H_i,S_{x_i})$, where $I$ has the cardinality of the number of connected components of $X$, $S_{x_i}\in Inn(X)$ is the inner automorphism of $X$ corresponding to $x_i\in X$ and $H_i= \mathrm{Stab}_{\Inn(X)}(x_i)$ is the stabilizer of $x_i$ in $\Inn(X)$.
\end{prop}

\begin{proof}
Note that there is a natural left action of $\Inn(X)$ on $X$. This gives a right action of $\Inn(X)$ on $X$ defined as $x \cdot g=g^{-1}(x)$ for $x \in X$ and $g \in \Inn(X)$. Let $\{X_i \mid i \in I \}$ be the set of orbits under this right action of $\Inn(X)$ on $X$. Note that each $X_i$ is a subquandle of $X$, and we can write $X=\sqcup_{i\in I} X_i$. Fix an $x_i\in X_i$ for each $i \in I$. Let $H_i= \mathrm{Stab}_{\Inn(X)}(x_i)$ be the stabilizer of $x_i$ in $\Inn(X)$. If $g \in H_i$, then $x_i \cdot g= x_i$, that is, $g^{-1}(x_i)=x_i$. This implies that $g^{-1}S_{x_i}g=S_{g^{-1}(x_i)}=S_{x_i}$, and hence $g\in C_{\Inn(X)}(S_{x_i})$. Thus, by Example \ref{lem:cosets_quandle 2}, we obtain the quandle $\sqcup_{i \in I}(\Inn(X)/H_i,S_{x_i})$. By the orbit-stabilizer theorem, for each $i \in I$, there is a bijective map $\phi_i:\Inn(X)/H_i\to X_i$ defined as $\phi_i(H_ig) = x_i \cdot g$. This gives a bijection $\phi:\sqcup_{i \in I}\Inn(X)/H_i \to X$ between the disjoint unions. We claim that $\phi$ is a quandle homomorphism. Since $S_{x_i}\in H_i$, we have 
\begin{eqnarray*}
\phi(H_ig\ast H_j h) &=& \phi(H_iS_{x_i} gh^{-1}S_{x_j}^{-1}h)\\
&=& \phi_i(H_igh^{-1}S_{x_j}^{-1}h)\\
&=& x_i \cdot (gh^{-1}S_{x_j}^{-1}h)\\
&=& (h^{-1}S_{x_j}h g^{-1})(x_i)\\
&=& (S_{h^{-1}(x_j)}g^{-1})(x_i)\\
&=& g^{-1}(x_i) * h^{-1}(x_j)\\
&=& (x_i \cdot g)\ast (x_j \cdot h)\\
&=& \phi_i(H_i g)\ast \phi_j(H_j h)\\
&=& \phi(H_i g)\ast \phi(H_j h)
\end{eqnarray*}
for all $g, h \in \Inn(X)$, and the proof is complete.
\end{proof}

\subsection{Bounded cohomology of groups}
Next, we recall the definition of bounded cohomology of groups \cite{MR3726870}. We first recall the construction of the cochain complex defining the group cohomology with real coefficients \cite{MR1324339}. Let $G$ be a group. For each $n \geq 0$, let $\Ca^n(G, \mathbb{R})$ be the real vector space with the basis consisting of maps $G^n \to \mathbb{R}$, where we take $G^0$ to be the trivial group. The coboundary map $$\partial^n:\Ca^n(G, \mathbb{R}) \to \Ca^{n+1}(G, \mathbb{R})$$ is defined by
\begin{equation}\label{eqn1}
\partial^n(\sigma)(g_1,\ldots,g_{n+1}) = \sigma(g_2,\ldots,g_{n+1})  +  \sum_{i=1}^n (-1)^{i}\sigma(g_1,\ldots,g_ig_{i+1},...,g_{n+1}) + (-1)^{n+1}\sigma(g_1,\ldots,g_n),
\end{equation}
for all $\sigma \in \Ca^n(G, \mathbb{R})$ and $(g_1,\ldots,g_{n+1}) \in G^{n+1}$. The cohomology $\Ha^*(G,\mathbb{R})$ of the cochain complex $\Ca^*(G,\mathbb{R})$ is defined to be the {\it group cohomology} of $G$ with coefficients in $\mathbb{R}$. Let $\Ca^*_b(G, \mathbb{R})$  be the subspace of bounded $n$-cochains with respect to the sup-norm, which is defined as $$\|f\|_\infty = \sup \{ |f(x)|~ \mid x \in G \}$$ for each map $f : G \to \mathbb{R}$. Then the cohomology $\Ha_b^*(G,\mathbb{R})$ of $\Ca^*_b(G,\mathbb{R})$ is called the {\it bounded group cohomology} of $G$ with coefficients in $\mathbb{R}$. We denote the group of bounded group cocycles and that of bounded group coboundaries by $\Za_b^*(G,\mathbb{R})$ and $\Ba^*_b(G,\mathbb{R})$, respectively. The inclusion $\Ca^*_b(G,\mathbb{R}) \to \Ca^*(G, \mathbb{R})$ induces the {\it comparison homomorphism} $$\ca^* : \Ha^*_b(G,\mathbb{R}) \to \Ha^*(G,\mathbb{R}).$$
\par

A \textit{group quasimorphism} on a group $G$ is a map $f: G \to \mathbb{R}$ such that
\begin{equation} \label{eq:defect}
D(f) := \sup_{g, h \in G} |f(g) +f(h) -f(gh)| < \infty.
\end{equation}
The constant $D(f)$ is called the \textit{defect} of $f$. In terms of group cohomology, $f$ is a group quasimorphism if and only if $\partial^1(f)$ is bounded with respect to the sup-norm. We denote the vector space of all group quasimorphisms on $G$ by $\mathcal{G}(G,\mathbb{R}).$
\par	

A group quasimorphism $f: G \to \mathbb{R}$ is called \textit{homogeneous} if it satisfies $f(g^n) = n \,f(g)$ for any $g \in G$ and $n \in \mathbb{Z}$. The subspace of all homogeneous group quasimorphisms on $G$ is denoted by $\mathcal{HG}(G, \mathbb{R}).$
\par

Let $\ca^* : \Ha^*_b(G,\mathbb{R}) \to \Ha^*(G,\mathbb{R})$ be the comparison homomorphism. By \cite[Corollary 2.11]{MR3726870}, we have
$$\mathcal{G}(G, \mathbb{R}) = \mathcal{HG}(G, \mathbb{R}) \oplus \Ca^1_b(G, \mathbb{R}) $$
		and
$$ \mathcal{HG}(G, \mathbb{R}) / \text{Hom}(G, \mathbb{R}) \cong \ker(\ca^2).$$
\medskip

\subsection{Bounded cohomology of quandles}
We now recall the definition of bounded cohomology of racks and quandles \cite{MR4779104}. Let $X$ be a rack and $n \ge 0$ an integer. Let $C^n_R(X, \mathbb{R})$ be the real vector space with the basis consisting of set-theoretic maps $X^n \to \mathbb{R}$, where $X^0$ is a singleton set. The coboundary operator $\delta^{n-1} : C^{n-1}_R(X, \mathbb{R}) \to C^{n}_R(X, \mathbb{R})$ is defined by
	\begin{equation} \label{eq:coboundary_operator}
		\delta^{n-1} f(x_1, \ldots, x_n) = \sum_{i=2}^{n} (-1)^i \big(f(x_1, \ldots, x_{i-1}, x_{i+1}, \ldots, x_n) - f(x_1* x_i, \ldots, x_{i-1}* x_i, x_{i+1}, \ldots, x_n)\big)
	\end{equation}
for all $f \in C^{n-1}_R(X, \mathbb{R})$ and $x_i \in X$. A routine check shows that $C^*_R(X, \mathbb{R}) = \{C^n_R(X, \mathbb{R}), \delta^n\}$ is a cochain complex, and the cohomology $H^n_R(X, \mathbb{R})$ of this cochain complex is called the {\it rack cohomology} of the rack $X$.
\par

In addition, if $X$ is a quandle, then we consider the subcomplex $C^*_Q(X, \mathbb{R})$ of $C^*_R(X, \mathbb{R})$ defined by
	\begin{equation} \label{eq:subcomplex}
		C^n_Q(X, \mathbb{R}) = \{ f \in C^{n}_R(X, \mathbb{R}) \mid f(x_1, \ldots, x_n) = 0~\text{whenever}~ x_i = x_{i+1}~ \text{for some}~ i \}
	\end{equation}
	for $n \geq 2$, and $C^n_Q(X, \mathbb{R}) = C^{n}_R(X, \mathbb{R})$ for $n=0, 1$. The cohomology of this subcomplex is called the {\it quandle cohomology} of $X$. We denote the rack and the quandle cohomology by $H^n_R(X, \mathbb{R})$ and $H^n_Q(X, \mathbb{R})$, respectively. We refer the reader to \cite[Section 3]{MR1990571} for details on quandle cohomology. 	
\par

The bounded cohomology of racks and quandles can be defined similarly to that of groups \cite{MR4779104}. If $X$ is a rack, then we have the sup-norm defined by $$\|f\|_\infty = \sup \{ |f(x)|~ \mid x \in X \}$$ 
for each map $f : X \to \mathbb{R}$. Let $C^n_{R, b}(X, \mathbb{R})$ be the subspace of $C^n_R(X, \mathbb{R})$ consisting of all maps $X^n \to \mathbb{R}$ which are bounded with respect to the sup-norm. The coboundary operators $\delta^n$ restrict to $C^n_{R, b}(X, \mathbb{R})$, and hence $C^*_{R, b}(X, \mathbb{R}) = \{C^n_{R, b}(X, \mathbb{R}), \delta^n\}$ forms a cochain complex. The cohomology $H^n_{R, b}(X, \mathbb{R})$ of this complex is called the \textit{bounded rack cohomology}. The natural inclusion $C^n_{R, b}(X, \mathbb{R}) \hookrightarrow C^n_R(X, \mathbb{R})$ induces a homomorphism $$c^n : H^n_{R, b}(X, \mathbb{R}) \to H^n_R(X, \mathbb{R}),$$ called the \textit{comparison homomorphism}. 
If $X$ is a quandle, we can similarly define the \textit{bounded quandle cohomology} and establish comparison maps. We denote the bounded cohomology of the quandle $X$ by $H^n_{Q,b}(X, \mathbb{R})$.
\medskip

\section{New quandle structures on surfaces}\label{section quandles on surfaces}

For an orientable surface $S$ possibly with boundary and/or marked points, let $\M(S)$ denote its mapping class group. For an integer $g\geq 1$, if $S_g$ is the closed orientable surface of genus $g$, we denote its {\it mapping class group} simply by $\mathcal{M}_g$. Throughout, we will use the same notation for a curve and its isotopy class, unless otherwise necessary. A curve is called {\it essential} if it is not homotopic to a point or to a puncture or to a boundary component of the surface. Up to isotopy, we assume that any two distinct curves on the surface are either disjoint or they intersect transversely.
\par

We let $\D_g$ denote the set of isotopy classes of unoriented essential simple closed curves on $S_g$. For each $\alpha \in \D_g$, let $T_\alpha$ denote the (right-handed) Dehn twist along $\alpha$. It is well-known that the group $\mathcal{M}_g$ is generated by such Dehn twists. Interestingly, the set $\D_g$ becomes a quandle, called the {\it Dehn quandle} of $S_g$, with respect to the binary operation given by $$\alpha \ast \beta=T_{\beta}(\alpha)$$ for $\alpha, \beta \in \D_g$. If $\D_g^{ns}$ denotes the set of  isotopy classes of essential unoriented non-separating simple closed curves on $S_g$, then it forms a subquandle of $\D_g$. It follows from \cite[Proposition 5.4]{MR4669143} that both $\D_g^{ns}$ and $\D_g$ are finitely generated. Moreover, by \cite[Proposition 6.4]{MR4669143}, $\D_g^{ns}$ is connected, whereas $\D_g$ is disconnected.
\par

In this section, we present generalisations of the preceding construction. Consider the set $\C_g$ of isotopy classes of all unoriented essential closed curves on $S_g$. For each $\alpha \in \C_g$, we choose a generic representative having the minimum number of transverse self-intersection points in its isotopy class. Such a representative will be called a {\it minimal representative}.
\par

 There is a unique way of resolving self-intersection points of an unoriented essential closed curve on $S_g$, which we describe next. We fix an orientation on the minimal representative of a closed curve $\alpha \in \C_g$. The local picture at every self-intersection point of $\alpha$ is shown on the left-hand side of Figure~\ref{fig:independence_of_resolution_on_the_orientation}.
\begin{figure}
\centering
\includegraphics[scale=1.5]{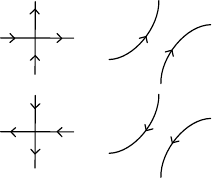}
\caption{Resolution of a self-intersection point of a non-simple closed curve and independence of the resultant multicurve on the chosen orientation.}
\label{fig:independence_of_resolution_on_the_orientation}
\end{figure} 
We then resolve these self-intersection points as shown on the right-hand side of Figure~\ref{fig:independence_of_resolution_on_the_orientation}. The same figure also shows that the outcome does not depend on the choice of the orientation on $\alpha$. This procedure of resolution of self-intersection points of non-simple closed curves is well-known, see, for example, \cite[Section 8]{MR1142906}. As an illustration, see Figure~\ref{fig:resolution_of_self_intersection_points} for the resolution of self-intersection points of a non-simple closed curve on the surface of genus two.
\begin{figure}
\centering
\includegraphics[scale=1.3]{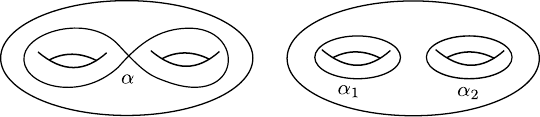}
\vspace{2mm}
\caption{Resolution of self-intersection points of a non-simple closed curve on $S_2$ into simple closed curves.}
\label{fig:resolution_of_self_intersection_points}
\end{figure}

By  \cite[Theorem 1]{MR1441263} and \cite[Theorem 2.1]{MR1259513}, any two minimal representatives of the same isotopy class of closed curves are related by ambient isotopies and third Reidemeister moves. It can be seen by resolving the three crossings on both the sides of the third Reidemeister move that the final smoothing is identical.
\par

Let $\alpha \in \C_g$ and $\alpha_1,\dots, \alpha_n$ be the simple closed curves obtained after resolving all the self-intersection points of $\alpha$. It could happen that some of these simple closed curves are isotopic. For instance, see Figure \ref{fig:resolution_of_nonsimple_curve_to_parallel_copies} for a non-simple closed curve on the torus whose resolution leads to two isotopic simple closed curves.
\begin{figure}
\centering
\includegraphics[scale=2.5]{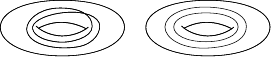}
\caption{Resolution of a closed curve leading to two isotopic simple closed curves.}
\label{fig:resolution_of_nonsimple_curve_to_parallel_copies}
\end{figure}
We refer to the multiset $\{\alpha_1,\dots, \alpha_n\}$ as the {\it associated multicurve} of $\alpha$.  Note the difference in our terminology from the traditional definition of a {\it multicurve}, which is a collection of pairwise disjoint essential simple closed curves. Given $\alpha \in \C_g$, we define 
$$T_{\alpha}=T_{\alpha_1}\cdots T_{\alpha_n},$$ 
where $\{\alpha_1,\dots,\alpha_n\}$ is the associated multicurve of $\alpha$. In other words, $T_{\alpha}$ is a multitwist being a finite product of powers of Dehn twists about disjoint essential simple closed curves.

\begin{prop}\label{Cg quandle}
For each $g \ge 1$, the set $\C_g$ forms a quandle with respect to the binary operation given by $$\beta \ast \alpha=T_{\alpha}(\beta)$$ 
for all $\alpha, \beta\in \C_g$.
\end{prop}

\begin{proof}
Note that $\alpha \ast \alpha=T_\alpha(\alpha)=\alpha$ for all $\alpha\in \C_g$, which is the first quandle axiom. For $\alpha,\beta\in \C_g$, we have $T_{\alpha} \, T_{\alpha}^{-1}(\beta)=\beta= T_{\alpha}^{-1}\,  T_{\alpha} (\beta)$. Hence, the map $S_{\alpha}:\C_g\to \C_g$ defined by $S_{\alpha}(\gamma)= \gamma \ast \alpha=T_{\alpha}(\gamma)$ is bijective with the inverse given by $S_{\alpha}^{-1}(\gamma)= T_{\alpha}^{-1}(\gamma)$, which establishes the second quandle axiom. Let $\alpha\in \C_g$, $f\in \mcg_g$ and $\{\alpha_1,\dots,\alpha_n\}$ be the associated multicurve of $\alpha$. Then the associated multicurve of $f(\alpha)$ is $\{f(\alpha_1),\dots, f(\alpha_n)\}$. This gives 
\begin{equation}\label{conjugation by mcg}
fT_{\alpha}f^{-1}=fT_{\alpha_1}\cdots T_{\alpha_n}f^{-1}=T_{f(\alpha_1)}\cdots T_{f(\alpha_n)}=T_{f(\alpha)}.
\end{equation}
Thus, we have $$(\gamma \ast \alpha)\ast (\beta\ast \alpha)=T_{\alpha}(\gamma)\ast T_{\alpha}(\beta)=T_{T_{\alpha}(\beta)}(T_{\alpha}(\gamma))=T_{\alpha}T_{\beta}(\gamma)=T_{\beta}(\gamma)\ast \alpha=(\gamma\ast\beta)\ast \alpha,$$
and hence $(\C_g,\ast)$ is a quandle. 
\end{proof}

\begin{cor}
\label{cor:cg_subquandle_of_mcg}
For each $g \ge 1$, the following assertions hold:
\begin{enumerate}
\item $\D_g$ is a subquandle of $\C_g$.
\item $\C_g$ is an infinitely generated quandle.
\item $\C_g$ has infinitely many connected components.
\end{enumerate}
\end{cor}

\begin{proof}
For assertion (1), note that $\D_g$ is a subset of $\C_g$ and the quandle operation in $\C_g$ coincides with that of $\D_g$ when restricted to simple closed curves.
\par

For $\alpha\in \C_g$, let $\{\alpha_1,\dots, \alpha_n\}$ be the associated multicurve of $\alpha$. In view of \eqref{conjugation by mcg}, we see that the map $\alpha \mapsto T_{\alpha}=T_{\alpha_1}\cdots T_{\alpha_n}$ gives a quandle homomorphism $\C_g \to \Conj(\mcg_g)$, where the latter is the conjugation quandle of $\mcg_g$. Figure~\ref{fig:two-nonisotopic_copies_of_figure_eight_curve} shows that the homomorphism is not injective. However, a quotient $\overline{\C}_g$ of $\C_g$ can be seen as a subquandle of $\Conj(\mcg_g)$. Let $\alpha \in \C_g$ be a non-trivial curve with associated multicurve $\{\alpha_1,\dots, \alpha_n\}$ such that $\alpha_i=\alpha_j$ for all $i\neq j$. Then $\alpha$ corresponds to $T_{\alpha_1}^n\in \Conj(\mcg_g)$. For $n\neq m$, since $T_{\alpha_1}^n$ is not conjugate to $T_{\alpha_1}^m$ in $\mcg_g$ \cite[Section 3.3]{MR2850125}, we conclude that $\overline{\C}_g$ is not a finitely generated quandle. Consequently, $\C_g$ is not a finitely generated quandle, which proves assertion (2).
\par

Note that connected components of a subquandle of $\Conj(\M_g)$ correspond to conjugacy classes in $\M_g$. For a simple closed curve $\alpha$, since $T_{\alpha}^m$ is not conjugate to $T_{\alpha}^n$ for $m\neq n$, it follows that $\overline{\C}_g$ has infinitely many connected components. Consequently, $\C_g$ has infinitely many connected components, which proves assertion (3).
\end{proof}

Next, we show that the construction of $\D_g$ can also be generalised to the set of integral weighted multicurves. Let $\E_g$ be the set of all multicurves on the surface $S_g$ and
$$\W_g=\left\{\sum_{i=1}^r n_i\alpha_i~\Big|~n_i\in \Z \setminus \{ 0\}, ~ \{\alpha_1, \ldots, \alpha_r \}  \in  \E_g \text{ and } r \ge 1 \right\}$$
the set of integral weighted multicurves.
\par

\begin{prop}\label{Dg quandle}
For each $g \ge 1$, the set $\W_g$ forms a quandle with respect to the binary operation given by 
$$\left(\sum_{i=1}^r n_i\alpha_i\right)\ast \left(\sum_{j=1}^s m_j\beta_j\right)=\sum_{i=1}^r n_iT_{\beta_1}^{m_1}\cdots T_{\beta_s}^{m_s}(\alpha_i).
$$
\end{prop}

\begin{proof}
For the first quandle axiom, we see that $$\left(\sum_{i=1}^r n_i\alpha_i\right)\ast \left(\sum_{j=1}^r n_j\alpha_j\right)=\sum_{i=1}^r n_iT_{\alpha_1}^{n_1}\cdots T_{\alpha_r}^{n_r}(\alpha_i)=\sum_{i=1}^r n_i\alpha_i$$ for each $\textstyle\sum_{i=1}^r n_i\alpha_i\in \W_g$. For $\textstyle\sum_{i=1}^r n_i\alpha_i$ and $\textstyle\sum_{j=1}^s m_j\beta_j$ in $\W_g$, we have $$\left(\sum_{i=1}^r n_i T_{\beta_1}^{-m_1}\cdots T_{\beta_s}^{-m_s}(\alpha_i)\right)\ast \left(\sum _{j=1}^s m_j\beta_j\right)=\sum_{i=1}^rn_i\alpha_i.$$
Furthermore, if $\textstyle\left(\sum_{k=1}^tp_k\gamma_k \right) \ast \textstyle\left(\sum _{j=1}^s m_j\beta_j\right)=\textstyle\sum_{i=1}^rn_i\alpha_i$, then
$$
\sum _{k=1}^t p_kT_{\beta_1}^{m_1}\cdots T_{\beta_s}^{m_s} (\gamma_k)=\sum _{i=1}^r n_i\alpha_i.
$$
This implies that $r=t$. After relabelling, if necessary, we obtain $p_i=n_i$ and $T_{\beta_1}^{m_1}\cdots T_{\beta_s}^{m_s} (\gamma_i)=\alpha_i$ for each $i$. Hence, we have
$$\sum_{k=1}^tp_k\gamma_k=\sum_{i=1}^r n_i T_{\beta_1}^{-m_1}\cdots T_{\beta_s}^{-m_s}(\alpha_i),$$ 
which establishes the second quandle axiom.  Finally, we have
\begin{align*}
&\left(\left(\sum_{k=1}^t p_k\gamma_k\right)\ast\left(\sum_{i=1}^r n_i\alpha_i\right)\right)\ast\left(\left(\sum_{j=1}^s m_j\beta_j\right)\ast \left(\sum_{i=1}^r n_i\alpha_i\right)\right)\\
&=\left(\sum_{k=1}^t p_k T_{\alpha_1}^{n_1}\cdots T_{\alpha_r}^{n_r}(\gamma_k)\right)\ast\left(\sum_{j=1}^s m_jT_{\alpha_1}^{n_1}\cdots T_{\alpha_r}^{n_r}(\beta_j)\right)\\
&=\sum_{k=1}^t p_k T_{\alpha_1}^{n_1}\cdots T_{\alpha_r}^{n_r}T_{\beta_1}^{m_1}\cdots T_{\beta_s}^{m_s}(\gamma_k), \quad \textrm{using \eqref{conjugation by mcg} repeatedly}\\
&=\left(\sum_{k=1}^t p_k T_{\beta_1}^{m_1}\cdots T_{\beta_s}^{m_s}(\gamma_k)\right)\ast \left(\sum_{i=1}^rn_i\alpha_i\right)\\
&=\left(\left(\sum_{k=1}^t p_k\gamma_k\right)\ast\left(\sum_{j=1}^s m_j\beta_j\right)\right)\ast\left(\sum_{i=1}^r n_i\alpha_i\right),
\end{align*}
which is the third quandle axiom.
\end{proof}

\begin{cor}
\label{cor:zeg_subquandle_mcg}
For each $g \ge 1$, the following assertions hold:
\begin{enumerate}
\item $\D_g$ is a subquandle of $\W_g$. 
\item $\W_g$ is an infinitely generated subquandle of $\Conj(\mcg_g)$. 
\item $\W_g$ has infinitely many connected components.
\end{enumerate}
\end{cor}

\begin{proof}
Note that each element of $\D_g$ can be considered as an element of $\W_g$. Since the quandle operation on the two sets is identical, $\D_g$ is a subquandle of $\W_g$. Further, the map $\W_g \to \Conj(\mcg_g)$ given by $\textstyle\sum_{i=1}^r n_i\alpha_i \mapsto T_{\alpha_1}^{n_1}\cdots T_{\alpha_r}^{n_r}$ is an embedding of quandles. Hence, $\W_g$ can be regarded as a subquandle of $\Conj(\mcg_g)$. For a simple closed curve $\alpha$, since $T_{\alpha}^n$ is not conjugate to $T_{\alpha}^m$ for $n\neq m$ \cite[Section 3.3]{MR2850125}, we conclude that $\W_g$ is not finitely generated as a quandle and has infinitely many connected components.

\end{proof}

We now relate the preceding two constructions through the idea of a quandle covering, introduced by Eisermann \cite{MR1954330, MR3205568}. A quandle homomorphism $p: X  \to Y$ is called a {\it quandle covering} if it is surjective and $S_x=S_{x'}$ whenever $p(x) = p(x')$ for any $x, x' \in X$.

\begin{example}
Some basic examples of quandle coverings are:
\begin{enumerate}
\item A surjective group homomorphism $p: G \to H$ yields a quandle covering $\Conj(G) \to \Conj(H)$ if and only if $\ker(p)$ is a central subgroup of $G$.
\item Let $X$ be a quandle and $F$ a non-empty set viewed as a trivial quandle. Consider $X \times F$ with the product quandle structure $(x,s) * (y,t) = (x *y,s)$. Then the projection $p: X \times F \to X$ given by $(x,s) \to x$ is a quandle covering.
\end{enumerate}
\end{example}

Consider the subquandle
$$
\W_g^+=\left\{ \sum_{i=1}^r n_i\alpha_i ~\Big| ~n_i>0, ~\{\alpha_1, \dots,\alpha_r\}\in \E_g \text{ and } r\geq 1 \right\}
$$
of $\W_g$ consisting of positively weighted multicurves. The proof of Corollary \ref{cor:zeg_subquandle_mcg} also yields the following result.

\begin{cor} \label{cor:positive_multicurve_components}
For each $g \ge 1$, the following assertions hold:
\begin{enumerate}
\item $\D_g$ is a subquandle of $\W_g^+$. 
\item $\W_g^+$ is an infinitely generated subquandle of $\Conj(\mcg_g)$.
\item $\W_g^+$ has infinitely many connected components.
\end{enumerate}
\end{cor}

\begin{theorem}\label{theorem covering}
For each $g \ge 1$, the map $\phi:\C_g\to \W_g^+$ given by $\phi(\alpha)=\textstyle\sum_{i=1}^r \alpha_i$, where $\{\alpha_1,\dots, \alpha_r\}$ is the associated multicurve of $\alpha$, is a quandle covering. Further, $\phi|_{\D_g}:\D_g\to \D_g$ is the identity map for each $g \ge 1$, $\phi$ is not injective for $g \ge 2$, and it is an isomorphism for $g=1$.
\end{theorem}

\begin{proof}
For $\alpha,\beta\in \C_g$, let $\{\alpha_1,\dots,\alpha_r\}$ and $\{\beta_1,\dots,\beta_s\}$ be the associated multicurves of $\alpha$ and $\beta$, respectively. Then we have
$$\phi(\beta \ast \alpha)=\phi \left(T_{\alpha_1}\cdots T_{\alpha_r}(\beta) \right)=\sum_{j=1}^s T_{\alpha_1}\cdots T_{\alpha_r}(\beta_j)=\left(\sum_{j=1}^s \beta_j\right)\ast \left(\sum_{i=1}^r\alpha_i\right)=\phi(\beta)\ast \phi(\alpha),$$ and $\phi$ is a quandle homomorphism. 
\par

Let $\textstyle\sum_{i=1}^r n_i\alpha_i\in \W_g^+$ be a weighted multicurve. Given any disjoint simple closed curves $\alpha_i$ and $\alpha_j$, we can construct a non-simple closed curve with one self-intersection such that the resolution of this self-intersection gives back $\alpha_i$ and $\alpha_j$. See Figure~\ref{fig:band_connected_sum} for the construction. In fact, by taking two disjoint copies of a simple closed curve $\alpha_i$, the same construction yields a non-simple closed curve with one self-intersection such that the resolution of this self-intersection yields two copies of $\alpha_i$. Thus, given $\textstyle\sum_{i=1}^s n_i\alpha_i$, by repeating the construction, we obtain a non-simple curve $\alpha$ such that $\phi(\alpha)=\textstyle\sum_{i=1}^s n_i\alpha_i$. Hence, $\phi$ is a surjection onto $\W_g^+$.
\begin{figure}
\centering
\includegraphics[scale=2]{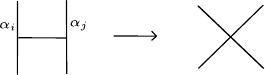}
\caption{Connecting two simple closed curves $\alpha_i$ and $\alpha_j$ to construct a non-simple closed curve whose resolution gives back $\alpha_i$ and $\alpha_j$.}
\label{fig:band_connected_sum}
\end{figure}
\par

Let $\alpha,\beta \in \C_g$ with associated multicurves $\{\alpha_1,\dots,\alpha_r\}$ and $\{\beta_1,\dots,\beta_s\}$, respectively.
If $\phi(\alpha)=\phi(\beta)$, then we have $\textstyle\sum_{i=1}^r \alpha_i=\textstyle\sum_{j=1}^s \beta_j$, which implies that $r=s$ and $\{\alpha_1,\dots,\alpha_r\}=\{\beta_1,\dots,\beta_s\}$. This gives $T_{\alpha}=T_{\alpha_1}\cdots T_{\alpha_r}=T_{\beta_1}\cdots T_{\beta_s}=T_{\beta}$. Hence, $\phi$ is a quandle covering onto $\W_g^+$.
\par

By definition, $\phi|_{\D_g}:\D_g\to \D_g$ is the identity map for each $g \ge 1$. Suppose that $g \ge 2$ and let $\iota$ be the hyperelliptic involution on $S_g$. If $\alpha$ is the figure eight curve on $S_g$ as shown in Figure~\ref{fig:two-nonisotopic_copies_of_figure_eight_curve}, then $\alpha \neq \iota(\alpha)$, but $\phi(\alpha)=\phi(\iota(\alpha))$. Hence, $\phi$ is not injective for $g \ge 2$. 
\begin{figure}
\centering
\includegraphics[scale=1.3]{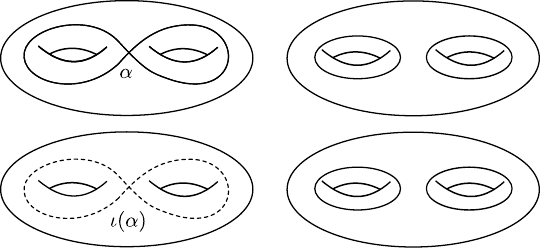}
\caption{Two non-isotopic curves with identical associated multicurves.}
\label{fig:two-nonisotopic_copies_of_figure_eight_curve}
\end{figure}
For the last assertion, note that any two essential simple closed curves on the torus intersect. Further, the isotopy class of an unoriented closed curve on the torus corresponds to $[p,q]\in (\Z \times \Z)/\{\pm 1\}$, while the isotopy class of an unoriented essential simple closed curve corresponds to $[p,q]\in (\Z \times \Z)/\{\pm 1\}$ with either $\gcd(p,q)=1$, or $(p,q)=(1,0)$, or $(p,q)=(0,1)$~\cite[Proposition 1.5]{MR2850125}. For any $[0,0]\neq [p,q]\in  (\Z \times \Z)/\{\pm 1\}$, we can write $ (p,q)=\gcd(p,q) (p',q')$, where $\gcd(p',q')=1$. Thus, the map $\phi$ sends the $[p,q]$ curve to $\gcd(p,q)[p',q']$, and it is an isomorphism in the case of the torus.
\end{proof}
\bigskip

Consider the subquandle
$$
\W_g^1=\left\{\sum_{i=1}^r \alpha_i ~\Big|~ \{\alpha_1,\ldots,\alpha_r\} \in \E_g \text{ and } r\ge 1\right\}
$$
of $\W_g$ consisting of weight one multicurves. Obviously, $\D_g$ is a subquandle of $\W_g^1$ for each $g\geq 1$. Though the quandle $\W_g$ is not finitely generated, we will see that $\W_g^1$ is  finitely generated. To prove this, we need a description of their inner automorphism groups.

\begin{prop} \label{prop:inner_automorphisms_of_quandles}
If $g\geq 1$ and $Z(\mcg_g)$ denotes the center of the mapping class group $\mcg_g$, then 
$$\Inn(\D_g^{ns}) \cong \Inn(\D_g) \cong \Inn(\W_g) \cong \Inn(\W_g^1)  \cong \Inn(\W_g^+) \cong \mcg_g/Z(\mcg_g).$$
\end{prop}

\begin{proof}
Let $X= \D_g^{ns}, \D_g, \W_g, \W_g^1$ or $\W_g^+$. By Corollary \ref{cor:zeg_subquandle_mcg}, we know that $\D_g^{ns}$ is a subquandle of $X$, which itself is a subquandle of $\Conj(\mcg_g)$. It is known that Dehn twists along elements of $\D_g^{ns}$ generate the mapping class group $\mcg_g$ \cite[Theorem 4.1]{MR2850125}. It follows that $X$ also generates $\mcg_g$. Thus, by~\cite[Proposition 3.21]{MR4669143}, we obtain $\Inn(X) \cong \mcg_g/Z(\mcg_g)$.
\end{proof}

\begin{prop}
For $g\geq 2$, we have $\Inn(\C_g)\cong\mcg_g$, whereas $\Inn(\C_1)\cong\mcg_1/Z(\mcg_1)$.
\end{prop}

\begin{proof}
Let $\rho:\mathcal M_g\to\Aut(\C_g)$ be the natural action of $\mcg_g$ on $\C_g$ given by $\rho(f)(\alpha)=f(\alpha)$ for each $f \in \mcg_g$ and $\alpha \in \C_g$. We know that $\mcg_g$ is generated by Dehn twists along essential non-separating simple closed curves. Thus, if $f=T_{\alpha_1}^{\pm 1} \cdots T_{\alpha_n}^{\pm 1}$ for some $\alpha_i \in \D_g^{ns} \subset \C_g$, then $\rho(f)=S_{\alpha_1}^{\pm 1} \cdots S_{\alpha_n}^{\pm 1}$. Hence, $\rho(\mathcal M_g)=\langle S_\alpha \mid \alpha\in \C_g\rangle=\Inn(\mathcal C_g)$. If $f\in\ker(\rho)$, then $f(\alpha)=\alpha$ for each $\alpha\in\mathcal D_g^{ns}$. This implies that $fT_\alpha f^{-1}=T_{f(\alpha)}=T_\alpha$ for each $\alpha\in\mathcal D_g^{ns}$. Hence, $f\in Z(\mathcal M_g)$, and therefore $\ker(\rho)\le Z(\mcg_g)$.
\par
By \cite[Theorem 3.10]{MR2850125}, $Z(\mcg_g)$ is trivial for $g\geq 3$, and hence $\Inn(\C_g)\cong\mcg_g$ for $g \ge 3$. On the other hand, for $g=1,2$, we know that $Z(\mcg_g)=\langle \iota \rangle\cong \Z_2$ \cite[p. 77]{MR2850125}, where $\iota$ is the hyperelliptic involution on $S_g$. Figure~\ref{fig:two-nonisotopic_copies_of_figure_eight_curve} shows that $\iota$ does not act trivially on $\C_2$. Thus, $\ker(\rho)$ is trivial, and hence $\Inn(\C_2)\cong\mcg_2$. For $g=1$, we can identify essential closed curves in $\C_1$ with equivalence class $[0,0]\neq[p,q]\in (\Z\times \Z)/\{\pm 1\}$ \cite[p. 23]{MR2850125}. Further, we have $\mcg_1=\mathrm{SL}(2,\Z)$ ~\cite[Theorem 2.5]{MR2850125}, and hence $\iota$ is a linear map. Since $\iota ([1,0])=[1,0]$ and $\iota[0,1]=[0,1]$, we have $\iota [p,q]=[p,q]$. Thus, $\iota$ acts trivially on $\C_1$, that is, $\iota \in \ker(\rho)$, and hence $\Inn(\C_1) \cong \mcg_1/Z(\mcg_1)$.
\end{proof}

The automorphism groups of $\D_g^{ns}$ and $\D_g$ are known due to~\cite[Theorem 6.3]{MR4669143}. This, together with the preceding result, naturally leads to the following problem.

\begin{prob}
Determine the automorphism groups of the quandles $\C_g$, $\W_g$, $\W_g^1$ and $\W_g^+$ for each $g \ge 1$.
\end{prob}

It is known from \cite[Proposition 6.5]{MR4669143} that $\D_g$ is a complete invariant of $S_g$ for $g\geq 3$. A similar result holds for the newly introduced quandles.

\begin{prop}\label{classifying surfaces}
For $g,g'\geq 3$, the following statements hold.
\begin{enumerate}
\item $S_g$ is homeomorphic to $S_{g'}$ if and only if $\C_g \cong \C_{g'}$.
\item $S_g$ is homeomorphic to $S_{g'}$ if and only if $\W_g \cong \W_{g'}$.
\item $S_g$ is homeomorphic to $S_{g'}$ if and only if $\W_g^1 \cong\W_{g'}^1$.
\item $S_g$ is homeomorphic to $S_{g'}$ if and only if $\W_g^+ \cong \W_{g'}^+$.
\end{enumerate}
\end{prop}

\begin{proof}
The forward implication in each case is obvious. For the converse, let $X= \C_g, \W_g, \W_g^1$ or $\W_{g}^+$and $X'= \C_{g'}, \W_{g'}, \W_{g'}^1$ or $\W_{g'}^+$ respectively, such that $X \cong X'$. Since $g, g' \ge 3$, it follows from \cite[Section 3.4]{MR2850125} and Proposition \ref{prop:inner_automorphisms_of_quandles} that 
$$\mcg_g \cong \Inn(X)\cong \Inn(X') \cong \mcg_{g'}.$$ Now, by \cite[Theorem A.1]{MR2805069}, we deduce that $S_g$ is homeomorphic to $S_{g'}$. 
\end{proof}

\begin{lem}\label{lem Eg Wg1 iso}
The map $\E_g \to \W_g^1$, given by $\{\alpha_1, \ldots, \alpha_r  \} \mapsto \textstyle\sum_{i=1}^r \alpha_i$, is an isomorphism of quandles.
\end{lem}

\begin{proof}
Consider the binary operation on $\E_g$ given by
$$
\{\alpha_1,\ldots,\alpha_r\}\ast \{\beta_1,\ldots,\beta_s\}=\{T_{\beta_1}\cdots T_{\beta_s}(\alpha_1),\ldots, T_{\beta_1}\cdots T_{\beta_s}(\alpha_r)\}.
$$
A direct check shows that $(\E_g,\ast)$ is a quandle. Further, the map $\phi:\E_g \to \W_g^1$ given by 
$$\{\alpha_1,\ldots,\alpha_r\} \mapsto \textstyle\sum_{i=1}^r \alpha_i$$ is an isomorphism of quandles, with its inverse $\W_g^1\to \E_g$ given by $$\textstyle\sum_{i=1}^r \alpha_i \mapsto \{\alpha_1,\ldots,\alpha_r\}.$$
\end{proof}

\begin{prop}\label{lem:zeg1_has_finite_components}
For each $g\geq 1$, the quandle $\W_g^1$ has finitely many connected components.
\end{prop}
\begin{proof}
The connected components of $\W_g^1$ are the orbits under the action of $\Inn(\W_g^1) =\mcg_g/Z(\mcg_g)$ on $\W_g^1$. For $g\geq 3$, by \cite[Theorem 3.10]{MR2850125}, we have $Z(\mcg_g)=1$. Hence, for $g\geq 3$, in view of Lemma \ref{lem Eg Wg1 iso}, it is enough to prove that $\mcg_g$ has finitely many orbits of multicurves on $S_g$. Let $C$ be a multicurve on $S_g$. Suppose that a connected component of the cut surface $\overline{S_g\setminus C}$ is a sphere with one boundary component. Since a sphere with one boundary component is a disk, there exists a curve $\gamma\in C$ which bounds a disk. But, such a curve $\gamma$ is null-homotopic, which is not possible. Now, suppose that a connected component of $\overline{S_g\setminus C}$ is a sphere with two boundary components. This means that there are two curves $\gamma_1,\gamma_2\in C$ such that $\gamma_1$ and $\gamma_2$ bound an annulus. But, such a pair of curves are isotopic, which is not possible. Hence, the set of connected components of $\overline{S_g\setminus C}$ does not contain a sphere with one or two boundary components. This implies that the Euler characteristic of any connected component of $\overline{S_g\setminus C}$ is at most $-1$. Since the Euler characteristic of $S_g$ is $2-2g$, there are at most $2g-2$ connected components of $\overline{S_g\setminus C}$. Thus, there are only finitely many homeomorphism types of connected components of $\overline{S_g\setminus C}$. Consequently, there are only finitely many choices of gluing the boundaries of connected components of $\overline{S_g\setminus C}$ to retrieve $S_g$. For each $\gamma\in C$, we denote the two resultant boundary curves of the cut-surface $\overline{S_g\setminus C}$ by $\gamma^1$ and $\gamma^2$. Given two multicurves $C_1$ and $C_2$ on $S_g$ such that $\phi(C_1)=C_2$ for some orientation-preserving homeomorphism $\phi$ of $S_g$, there is an induced orientation-preserving homeomorphism $\bar{\phi}$ of $\overline{S_g\setminus C_1}$ onto $\overline{S_g\setminus C_2}$. Note that, for each $\gamma \in C_1$, the resultant boundary curves $\gamma^1$ and $\gamma^2$ lie in the same connected component of $\overline{S_g\setminus C_1}$ if and only if both $\bar{\phi}(\gamma^1)$ and $\bar{\phi}(\gamma^2)$ lie in the same connected component of $\overline{S_g\setminus C_2}$. Conversely, if $\psi:\overline{S_g\setminus C_1}\to\overline{S_g\setminus C_2}$ is an orientation-preserving homeomorphism such that both $\overline{S_g\setminus C_1}$ and $\overline{S_g\setminus C_2}$ have the same gluing data, then the map $\psi$ extends to an orientation-preserving homeomorphism of $S_g$ that maps $C_1$ onto $C_2$. Thus, up to isotopy, two multicurves $C_1$ and $C_2$ on $S_g$ lie in the same orbit under the action of $\mcg_g$ if and only if $\overline{S_g\setminus C_1}$ is homeomorphic to $\overline{S_g\setminus C_2}$ and they admit the same gluing data. Since there are only finitely many choices for the gluing data, it follows that $\M_g$ has finitely many orbits of multicurves on $S_g$. 
\par

For $g=1$, we have $\W_1^1=\D_1$. Since every essential simple closed curve on $S_1$ is non-separating, they all lie in the same connected component of $\W_1^1$. Thus, $\W_1^1$ has precisely one connected component.
\par

For $g=2$, the size of any multicurve on $S_2$ is at most three. Note that there are no two disjoint essential separating curves on $S_2$. Let $C$ be a multicurve on $S_2$ containing only essential curves. If the size of $C$ is one, then it either contains a non-separating or a separating curve. Thus, we have two $\mcg_2$ orbits of multicurves of size one.
\par
Next, assume that $C$ contains two non-separating curves. Then $\overline{S_2\setminus C}$ is homeomorphic to a sphere with four boundary components. If $C_1$ and $C_2$ are two such multicurves, then $\overline{S_2\setminus C_1}$ is homeomorphic to $\overline{S_2\setminus C_2}$. Thus, there is a mapping class $f\in \mcg_2$ such that $f(C_1)=C_2$. Now, let $C$ contain one non-separating and one separating curve. Then $\overline{S_2\setminus C}$ is homeomorphic to a disjoint union of a pair of pants and a torus with one boundary component. Note that there is only one choice of gluing boundaries of a pair of pants to a torus with one boundary component to retrieve $S_2$. Thus, if there are two such multicurves $C_1$ and $C_2$, then $\overline{S_2\setminus C_1}$ is homeomorphic to $\overline{S_2\setminus C_2}$, and hence there is a mapping class $f\in \mcg_2$ such that $f(C_1)=C_2$. Hence, we have two $\mcg_2$ orbits of multicurves of size two. 
\par
Now, assume that the size of $C$ is three. In this case, $\overline{S_2\setminus C}$ is homeomorphic to a disjoint union of two pairs of pants. Note that there are two choices of gluing boundaries of these two pairs of pants to retrieve $S_2$. Based on these gluings, the resultant multicurve on $S_2$ either contains a separating curve or does not. Thus, there are two $\mcg_2$ orbits of multicurves of size three on $S_2$. Hence, $\mcg_2$ has  six orbits on $\W_2^1$. Figure~\ref{fig:six_orbits_of_multicurves_on_s2} shows one representative of each of these orbits. It is known from \cite[Section 3.4, Proposition 7.15]{MR2850125} that $Z(\mcg_2)$ is generated by the hyperelliptic involution $\iota$ on $\mcg_2$. Since these multicurves are invariant under $\iota$, we conclude that $\Inn(\W_2^1)$ also has six orbits on $\W_2^1$.
\end{proof}

\begin{figure}
\centering
\includegraphics[scale=1]{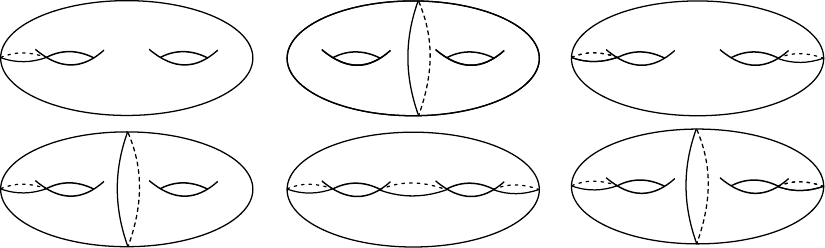}
\caption{Representatives of $\mcg_2$ orbits of essential multicurves on $S_2$.}
\label{fig:six_orbits_of_multicurves_on_s2}
\end{figure}

\begin{prop}\label{prop Wg1 is finitely generated}
For each $g\geq 1$, the quandle $\W_g^1$ is finitely generated. 
\end{prop}

\begin{proof}
Let $\D_g^{ns}$ be the subquandle of $\D_g$ consisting of isotopy classes of non-separating simple closed curves on $S_g$.  Then $\D_g^{ns}$ is also a subquandle of $\W_g^1$. By \cite[Section 1.3.1]{MR2850125}, for each pair of elements of $\D_g^{ns}$, there is a homeomorphism mapping one onto the other. Since such a homeomorphism is a product of Dehn twists along non-separating simple closed curves, it follows that elements of $\D_g^{ns}$ correspond to a single conjugacy class under the embedding $\D_g \hookrightarrow \Conj(\M_g)$. Consequently, the quandle $\D_g^{ns}$ is connected, and is a connected component of $\W_g^1$. By Proposition~\ref{lem:zeg1_has_finite_components}, $\W_g^1$ has finitely many connected components. We write $\W_g^1=\D_g^{ns} \sqcup_{i=1}^n D_i$, where each $D_i$ is a connected component of $\W_g^1$ other than $\D_g^{ns}$. It is known from \cite[Proposition 5.4]{MR4669143} that $\D_g^{ns}$ is a finitely generated quandle, in fact, generated by Humphries generators. Further, due to the embedding $\W_g^1 \hookrightarrow \Conj(\M_g)$, we can view each $D_i$ as a subset of a conjugacy class in $\M_g$. But, each element of $\M_g$ can be written as a product of Dehn twists along elements from $\D_g^{ns}$. Thus, fixing a representative $x_i$ of the connected component $D_i$, it follows that $\W_g^1$ is generated by a finite generating set of $\D_g^{ns}$ and the set $\{x_1, \ldots, x_n \}$. Hence, $\W_g^1$ is finitely generated.
\end{proof}

\begin{rem}
{\rm The table below summarizes the results concerning finite generation and the number of connected components of the Dehn quandles.
\begin{center}
\begin{table}[ht]
\begin{tabular}{|c|c|c|c|c|c|c|}
 \hline 
  & $\D_g^{ns}$ & $\D_g$ & $\C_g$ & $\W_g$ & $\W_g^+$ & $\W_g^1$ \\ 
 \hline 
 Finite generation & Yes & Yes & No & No & No & Yes \\ 
 \hline 
 Connectedness & Yes & No & No & No & No & No \\ 
 \hline
 Number of connected components & Finite & Finite & Infinite & Infinite & Infinite & Finite \\
 \hline
 \end{tabular}
 \end{table}
\end{center}}
\end{rem}

\begin{rem}
\rm{For convenience, we refer to all the quandles $\D_g^{ns}, \D_g, \C_g, \W_g, \W_g^1$ and $\W_g^+$ as {\it Dehn quandles} of the surface $S_g$. In view of Corollaries \ref{cor:cg_subquandle_of_mcg} and \ref{cor:zeg_subquandle_mcg}, except $\C_g$, we will identify elements of these quandles with corresponding elements of $\M_g$. Under this identification, a connected component corresponds to a conjugacy class in $\M_g$. For $\alpha \in \C_g$, we retain the notation $T_\alpha$ for the associated multitwist, but we do not identify $\alpha$ with $T_\alpha$.}
\end{rem}

\begin{rem}
\rm{ Our constructions naturally extend to orientable surfaces with punctures and/or boundary components. However, in this paper, we restrict our attention to closed orientable surfaces. For treatments of quandle structures on orientable surfaces with punctures, we refer the reader to \cite{MR1865704, MR1985908}.}
\end{rem}
\medskip

\section{Metrics on Dehn quandles of surfaces}\label{section Metrics on Dehn quandles}

Let $X$ be a quandle. In \cite{MR4779104}, K\k{e}dra introduced a metric, which we refer to as the {\it quandle metric}, on each connected component of $X$ by defining
	$$d(x,y)=\mathrm{min}\{n \in \mathbb{N} \mid y=(((x*^{\pm 1} x_1)*^{\pm 1}  x_2)*^{\pm 1}  \cdots )*^{\pm 1}  x_n~\textrm{for some}~x_1, \ldots, x_n \in X\}.$$
We say that $X$ is {\it bounded} if the metric $d$ is bounded (has finite diameter) on each connected component of $X$. Otherwise, it is called {\it unbounded}. It turns out that fundamental quandles of non-trivial knots and free quandles are unbounded with respect to this metric \cite[Example 3.17]{MR4779104}. 

The following result establishes a  relation of the quandle metric with the second bounded cohomology of quandles \cite[Theorem 1.2]{MR4779104}, thereby making the study of the latter relevant.

\begin{theorem}\label{bounded iff comp map injective}
Let $X$ be a quandle with finitely many connected components. Then $X$ is bounded if and only if the kernel of the comparison map $c^2: H^2_{Q,b}(X, \mathbb{R}) \to H^2_Q(X, \mathbb{R})$ is trivial.
\end{theorem}

We begin with the following observation.

\begin{prop}\label{unbounded quandles}
For each $g\geq 1$, the quandles $\D_g^{ns}, \D_g,\C_g, \W_g, \W_g^1$ and $\W_g^+$ are unbounded with respect to the quandle metric.
\end{prop}

\begin{proof}
Let $\psi:\mcg_g\to \R$ be a non-zero homogeneous group quasimorphism that vanishes on $\stab_{\mcg_g}(\alpha) \cup \{T_{\alpha'}\mid \alpha'\in \D_g^{ns}\}$ for some $\alpha\in \D_g^{ns}$ and has defect $D_{\psi}>0$. Such a quasimorphism exists by Lemma~\ref{lem:vanishing_quasimorphisms1}. Since $\psi$ is non-zero, there exists $f\in \mcg_g$ such that $\psi(f)\neq 0$. Fix an integer $n>0$. Then,  we have
\begin{align*}
d(\alpha,f^n(\alpha))&=\min \{m \in \mathbb{N} \mid f^n(\alpha)=(((\alpha*^{\pm 1} \alpha_1)*^{\pm 1}  \alpha_2)*^{\pm 1}  \cdots )*^{\pm 1}  \alpha_m~\textrm{for some}~\alpha_i \in \D_g^{ns}\}\\
&=\min\{m\in \mathbb{N}\mid f^n(\alpha)=T_{\alpha_m}^{\pm 1}\cdots T_{\alpha_2}^{\pm 1}T_{\alpha_1}^{\pm 1}(\alpha)~\textrm{for some}~\alpha_i \in \D_g^{ns}\}.
\end{align*}
If $h\in \mcg_g$ such that $h(\alpha)=f^n(\alpha)$, then $f^{-n}h\in \stab_{\mcg_g}(\alpha)$. Hence, $h=f^nh'$ for some $h'\in \stab_{\mcg_g}(\alpha)$. By choice of $\psi$, we have $|\psi(h)-\psi(f^n)-\psi(h')|\leq D_{\psi}$, $\psi(f^n)=n\psi(f)$ and $\psi(h')=0$. This implies that $|\psi(h)|-n|\psi(f)|\geq -D_{\psi}$, and hence $|\psi(h)|\geq n|\psi(f)|-D_{\psi}$. Suppose that $d(\alpha,f^n(\alpha))=m$ and choose $\alpha_1, \ldots, \alpha_m \in \D_g^{ns}$ realising it. If $h=T_{\alpha_m}^{\pm 1}\cdots T_{\alpha_2}^{\pm 1}T_{\alpha_1}^{\pm 1}$, then $f^n(\alpha)= h(\alpha)$. Since $\psi(T_{\alpha_i})=0$ for every $i$, induction on $m$ shows that $|\psi(h)|\leq (m-1)D_{\psi}$. Consequently, we obtain $(m-1)D_{\psi}\geq n|\psi(f)|-D_{\psi}$, and hence $m\geq n|\psi(f)|/D_{\psi}$. Since $d(\alpha,f^n(\alpha))=m\geq  n|\psi(f)|/D_{\psi}$ and $|\psi(f)|/D_{\psi}>0$, it follows that $d(\alpha,f^n(\alpha)) \longrightarrow \infty$ as $n\longrightarrow \infty$. Hence, the quandle $D_g^{ns}$ is unbounded with respect to the quandle metric. A similar argument works for $X=\D_g,\C_g,\W_g,\W_g^1,\W_g^+$ by considering a non-zero homogeneous group quasimorphism on $\mcg_g$ that vanishes on $\stab_{\mcg_g}(x) \cup \{S_{x'}\mid x' \in X\}$ for some $x \in X$ and observing that each $S_{x'}$ is a multitwist.
\end{proof}

For $\alpha, \beta \in \D_g$ with $g\ge 1$, let $i(\alpha, \beta)$ denote their {\it  geometric intersection number}. Recall that the {\it curve complex} of $S_g$ is a graph with vertex set $\D_g$ such that there is an edge between the vertices $\alpha, \beta \in \D_g$ if $i(\alpha, \beta) = 0$ for $g \ge 2$ and $i(\alpha, \beta) = 1$ for $g = 1$. This equips $\D_g$ with a natural metric $d'$, which we refer to as the {\it curve complex metric}. Below we explore how the metrics $d$ and $d'$ are related.
\par

\begin{lem}\label{metric lemma 1}
Let $g\geq2$, and let $\alpha,\beta\in \D_g^{ns}$ be two distinct curves such that $i(\alpha,\beta)=0$. Then $$d'(\alpha,\beta)=1\qquad\text{and}\qquad d(\alpha,\beta)>1.$$
\end{lem}
\begin{proof}
The first equality follows from the definition of the curve complex metric. If $d(\alpha,\beta)=1$, then $\beta=T_\gamma^{\pm1}(\alpha)$ for some
		$\gamma\in \D_g^{ns}$. This implies that
		$$
		0=i(\alpha,\beta)
		=i\bigl(\alpha,T_\gamma^{\pm1}(\alpha)\bigr)
		=i(\alpha,\gamma)^2,
		$$
and hence $i(\alpha,\gamma)=0$. Consequently, $T_\gamma^{\pm1}(\alpha)=\alpha$, which contradicts $\alpha\neq\beta$. Hence, we must have $d(\alpha,\beta)>1$.
\end{proof}
	
\begin{lem}\label{metric lemma 2}
		Let $g\geq1$, and let $\alpha,\beta\in \D_g^{ns}$ such that $i(\alpha,\beta)=1$. Then the following assertions hold:
		\begin{enumerate}
			\item $d(\alpha,\beta)=1$.
			\item If $g\geq2$, then $d'(\alpha,\beta)=2$.
		\end{enumerate}
	\end{lem}
	\begin{proof}
		The braid relation gives
		$$
		T_{T_\alpha(\beta)}(\alpha)
		=T_\alpha T_\beta T_\alpha^{-1}(\alpha)
		=T_\alpha T_\beta(\alpha)
		=\beta,
		$$
and hence $d(\alpha,\beta)=1$.
		If $g\geq2$, choose $\gamma \in \D_g$ which is disjoint from $\alpha\cup\beta$. Since $i(\alpha,\beta)\neq0$, we have $d'(\alpha,\beta)\geq2$. Furthermore, $\alpha,\gamma,\beta$ form an edge path of length two in the curve complex, and hence $d'(\alpha,\beta)=2$.
	\end{proof}	

Given two metrics $d_1$ and $d_2$ on a set $X$, we write $d_1\leq d_2$ if $d_1(x,y)\leq d_2(x,y)$ for all $x,y\in X$. If, in addition, there exist $x,y\in X$ such that $d_1(x,y)<d_2(x,y)$, then we write $d_1< d_2$.

\begin{prop}\label{prop:comparison_of_distance}
If $g \ge 2$, then neither $d' \le d$ nor  $d \le d'$. Furthermore, if $g=1$, then $d< d'$. 
\end{prop}

\begin{proof}
Suppose first that $g\geq2$. Let $\alpha,\beta\in \D_g^{ns}$ such that $i(\alpha, \beta)=0$. By Lemma \ref{metric lemma 1}, we have
$$d'(\alpha,\beta)=1<d(\alpha,\beta),$$
and hence $d\leq d'$ does not hold.	Next,  let $\mu,\nu \in \D_g^{ns}$ such that $i(\mu,\nu)=1$. By Lemma \ref{metric lemma 2}, we have
		$$
		d(\mu,\nu)=1<2=d'(\mu,\nu),
		$$
and hence $d'\leq d$ does not hold.	Now suppose that $g=1$, and let $\alpha, \beta \in \D_1$. Let
		$$
		\alpha=\gamma_0,\gamma_1,\ldots,\gamma_n=\beta
		$$
		be a shortest edge path in the curve complex. Since
		$i(\gamma_{j-1},\gamma_j)=1$, Lemma \ref{metric lemma 2} gives
		$d(\gamma_{j-1},\gamma_j)=1$. Hence, we obtain
		$$
		d(\alpha,\beta)\leq n=d'(\alpha,\beta).
		$$
		To see that the inequality can be strict, choose $\beta,\gamma\in \D_1$ such that $i(\beta,\gamma)=2$ (see Figure~\ref{fig:four_curves_on_torus}). Put $\delta=T_\gamma(\beta)$. Then, we see that $d(\beta,\delta)=1$. Since $i(\beta,\delta)=i(\beta,\gamma)^2=4$, we have $d'(\beta,\delta)\geq2$.
\begin{figure}[t]
\centering
\includegraphics[scale=.8]{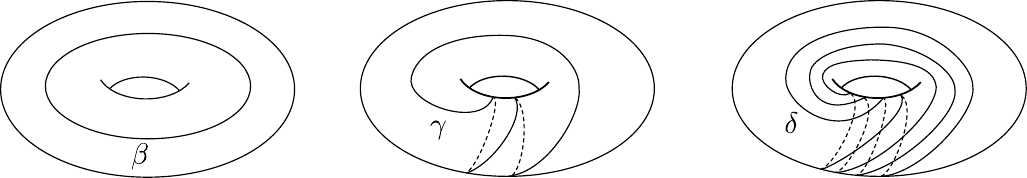}
\caption{Curves on the torus showing $d<d'$.}
\label{fig:four_curves_on_torus}
\end{figure}
Hence, the inequality is strict, that is, $d<d'$.
\end{proof}
\medskip

\section{Bounded cohomology of Dehn quandles and Gromov Mapping Theorem}\label{section Bounded cohomology of Dehn quandles}
In this section, we investigate the second bounded cohomology of Dehn quandles of surfaces discussed in Section \ref{section quandles on surfaces}, and establish an analogue of the Gromov Mapping Theorem for quandle extensions. 

\subsection{Bounded cohomology of Dehn quandles of surfaces}
It is known due to Bestvina and Fujiwara \cite[Theorem 12]{MR1914565} that for any subgroup $G$ of the mapping class group $\mcg_g$ which is not virtually abelian, the kernel of the comparison homomorphism $c^2:\Ha_b^2(G;\R)\to \Ha^2(G;\R)$ is infinite-dimensional. Since $\mcg_g$ is not virtually abelian, it follows that $\Ha_b^2(\mcg_g;\R)$ is also infinite-dimensional. In a follow-up paper, they proved the following generalization \cite[Theorem 1.1]{MR2332668}. 

\begin{theorem} \label{thm:bounded_morphisms_mcg}
Let $S$ be a compact surface such that $\M(S)$ is not virtually abelian. Then $\mathcal{HG}(\M(S);\R)$ is infinite-dimensional. Moreover, for every finite collection $G_1,\dots, G_n$
of cyclic subgroups of $\M(S)$, there is an infinite-dimensional subspace of $\mathcal{HG}(\M(S);\R)$ such that any element $\phi$ of this subspace satisfies:
\begin{enumerate}
\item $\phi$ is bounded on each $G_i$, and
\item $\phi$ is bounded on the stabilizer $\mathrm{Stab}_{\M(S)}(\alpha)$ of each essential simple closed curve $\alpha$ on $S$.
\end{enumerate}
\end{theorem}

It follows from the proof of \cite[Theorem 1.1]{MR2332668} that the subspace $\mathcal{HG}(\M(S);\R)$ has dimension equal to the cardinality of the continuum.

\begin{lem}\label{lem:vanishing_morphism_multitwist}
Let $g \ge 1$ and $\phi:\M_g\to \R$ be a group quasimorphism such that $\phi$ is bounded on the stabilizer of each essential simple closed curve on $S_g$. Then the homogeneous group quasimorphism $\tilde{\phi}:\M_g\to \R$ defined by
$$
\tilde{\phi}(g)=\lim_{n\to \infty} \frac{\phi(g^n)}{n}
$$
vanishes on the centralizer of each multitwist. Here, a {\it multitwist} is a product of integral powers of Dehn twists along curves constituting a multicurve.
\end{lem}

\begin{proof}
It follows from \cite[Proposition 2.10]{MR3726870} that $\tilde{\phi}$ is a homogeneous group quasimorphism. Let $T_{\alpha_1}^{n_1}\cdots T_{\alpha_r}^{n_r}$ be a multitwist, where $\{\alpha_1,\dots,\alpha_r\}$ is a multicurve and $n_i$'s are non-zero integers. If $f\in C_{\M_g}(T_{\alpha_1}^{n_1}\cdots T_{\alpha_r}^{n_r})$, then we have
$$T_{\alpha_1}^{n_1}\cdots T_{\alpha_r}^{n_r}= fT_{\alpha_1}^{n_1}\cdots T_{\alpha_r}^{n_r}f^{-1}=T_{f(\alpha_1)}^{n_1}\cdots T_{f(\alpha_r)}^{n_r}.$$
By~\cite[Lemma 3.17]{MR2850125}, the preceding equation yields $\{\alpha_1,\dots,\alpha_r\}=\{f(\alpha_1),\dots,f(\alpha_r)\}$. Then, there exists an integer $N>0$ such that $f^N(\alpha_i)=\alpha_i$ for each $i$, and hence $\langle f^N\rangle \le \mathrm{Stab}_{\M_g}(\alpha_i)$. We are given that $\phi$ is bounded on $\mathrm{Stab}_{\M_g}(\alpha_i)$ for each $i$. It follows that $\phi$ is bounded on $\langle f^N \rangle$, and hence there exists $K>0$ such that $|\phi(f^{Nm})| \le K$ for each integer $m$. Computing along the subsequence $\{Nm \mid m=1, 2, \ldots\}$, we get
$$
|\tilde{\phi}(f)|=\Big|\lim_{m\to \infty} \frac{\phi(f^{Nm})}{Nm}\Big|\leq \frac{K}{N}\lim_{m\to \infty} \frac{1}{m}=0.
$$
Hence, $\tilde{\phi}$ vanishes on the centralizer of each multitwist.
\end{proof}

\begin{lem} \label{lem:vanishing_quasimorphisms1}
Let $f$ be a non-trivial multitwist and $\{f_i\mid i\in I\}$ a family of multitwists. Then the space $$\mathcal{HG}^{C_{\M_g}(f)\cup \{f_i \,\mid \, i\in I\}}(\M_g;\R)$$ of homogeneous group quasimorphisms of $\M_g$ that vanish  on $\C_{\M_g}(f) \cup \{f_i\mid i\in I\}$ has dimension equal to the cardinality of the continuum.
\end{lem}

\begin{proof}
The result follows from Theorem~\ref{thm:bounded_morphisms_mcg} and Lemma~\ref{lem:vanishing_morphism_multitwist}.
\end{proof}

We are now in a position to present our main result on bounded cohomology of Dehn quandles.

\begin{theorem} \label{thm:D_g_cohomolofy_id}
Let $g\geq 1$ and $X=\D_g^{ns}, \D_g, \C_g, \W_g, \W_g^1$ or $\W_g^+$. Then $H_{Q, b}^2(X;\R)$ has dimension equal to the cardinality of the continuum.
\end{theorem}

\begin{proof}
By Proposition \ref{prop:inner_automorphisms_of_quandles}, we have $\Inn(X) \cong \mcg_g/Z(\mcg_g)$. In view of Proposition~\ref{prop:quandle_as_coset_quandle}, we have $X \cong \sqcup_{i\in I} (\Inn(X)/H_i,S_{x_i})$, where $i\in I$, $x_i\in X$, and $H_i=\mathrm{Stab}_{\Inn(X)}(x_i) = C_{\M_g}(S_{x_i})$. For each $x\in X$, we have $S_x=T_{\alpha_1}^{n_1}T_{\alpha_2}^{n_2}\cdots T_{\alpha_r}^{n_r}$ for some multicurve $\{\alpha_1,\dots,\alpha_r\}$ and some integers $n_i$. In other words,  $S_x$ is a multitwist for each $x \in X$. By Lemma~\ref{lem:vanishing_quasimorphisms1}, for any fixed $i_0 \in I$, the vector space $\mathcal{HG}^{H_{i_0}\cup \{S_{x_i}\mid i\in I\}}(\M_g;\R)$ has dimension equal to the cardinality of the continuum. Finally, it follows from \cite[Theorem 4.3 and Remark 4.4]{arXiv:2502.04069} that $H_{Q,b}^2(X;\R)$ has dimension equal to the cardinality of the continuum.
\end{proof}

\subsection{An analogue of the Gromov Mapping Theorem}
We recall the well-known Gromov Mapping Theorem for the bounded cohomology of groups \cite[Corollary 4.25]{MR3726870}.

\begin{theorem}
	Let $1 \longrightarrow H \longrightarrow G \longrightarrow K \longrightarrow 1$ 	be a short exact sequence of groups such that $H$ is amenable. Then the quotient map $G \longrightarrow K$ induces an isometric isomorphism $\Ha_b^n(K,\mathbb{R}) \longrightarrow \Ha_b^n(G,\mathbb{R})$ for each $n \ge 0$.
\end{theorem}

Recall that a group $A$ is called {\it amenable} if it admits a {\it left-invariant mean}. Here, a left-invariant mean on $A$ is a linear functional $m :  \Ca^1_b(A, \mathbb{R}) \longrightarrow \mathbb{R}$ satisfying the following properties:
	\begin{enumerate}
		\item Positivity: If $f \in  \Ca^1_b(A, \mathbb{R})$ is such that $f(a) \ge 0$ for all $a \in A$, then $m(f) \ge 0.$
		\item Normalization: $m(\mathbf{1}) = 1$, where $\mathbf{1} \in \Ca^1_b(A, \mathbb{R})$ denotes the constant map given by $\mathbf{1}(a)=1$ for all $a \in A$. 
		\item Left-invariance: If $b \in A$ and $f \in  \Ca^1_b(A, \mathbb{R})$, then $m(L_b f) = m(f)$, where $L_b f \in  \Ca^1_b(A, \mathbb{R})$ is defined by $(L_b f)(a) = f(ba)$ for all $a \in A$.
	\end{enumerate}

It follows from the definition that, if $f \in \Ca^1_b(A, \mathbb{R})$, then
$$	-\|f\|_\infty \, \mathbf{1} \le f \le \|f\|_\infty \, \mathbf{1}.$$
Further, positivity and linearity of the mean implies that
	\begin{eqnarray} \label{inequality}
		|m(f)| \le \|f\|_\infty.
	\end{eqnarray}

Let $A$ be an amenable group and $m : \Ca^1_b(A, \mathbb{R}) \to \mathbb{R}$ a left-invariant mean on $A$. Then $m$ inductively induces a family of means $m^{(n)} : \Ca^1_b(A^n, \mathbb{R}) \to \mathbb{R}$ on $A^n$ as follows:
\begin{itemize}
	\item $m^{(1)} := m$;
	\item for $n \ge 1$ and $f \in \Ca^1_b(A^{n+1}, \mathbb{R})$, we have
	\begin{eqnarray*}
		m^{(n+1)}(f)
		:=
		m\!\left(
		a_{n+1} \mapsto
		m^{(n)}\big(f(-,a_{n+1})\big)
		\right).
	\end{eqnarray*}
\end{itemize}
We see that $m^{(n)}$ is a left-invariant mean on $A^n$ for each $n \ge 1$.

\begin{lem}\label{mean on product}
	Let $n \ge 1$ be an integer, $A$ be an amenable group and $g \in \Ca^1_b(A^{n+1}, \mathbb{R})$. Suppose that there exist $h \in \Ca^1_b(A^n, \mathbb{R})$ and $1 \le i \le n+1$ such that
	\begin{eqnarray*}
		g(a_1,\dots,a_{n+1})=h(a_1,\dots,\widehat{a_i},\dots,a_{n+1})
	\end{eqnarray*}
	for all $(a_1,\dots,a_{n+1}) \in A^{n+1}$.
	Then $	m^{(n+1)}(g)
	=
	m^{(n)}(h).$
\end{lem}

\begin{proof}
	We argue by induction on $n$. For $n=1$, let $h \in \Ca^1_b(A, \mathbb{R})$ and $g \in \Ca^1_b(A^2, \mathbb{R})$ such that $g(a_1,a_2)=h(a_1)$ for all $(a_1,a_2) \in A^2$. By definition, we have
$$ m^{(2)}(g) =	 m^{(1)}\!\left(	a_1 \mapsto m^{(1)}\!\left( a_2 \mapsto g(a_1,a_2)\right) \right)= m^{(1)}\!\left(a_1 \mapsto m^{(1)}\!\left(a_2 \mapsto h(a_1)\right) \right).$$
Since $h(a_1)$ is constant with respect to $a_2$, we obtain
		\begin{eqnarray*}
			m^{(1)}\!\left(
			a_2 \mapsto h(a_1)
			\right)=h(a_1)\, m^{(1)}({\bf 1})=h(a_1).
		\end{eqnarray*}	
Hence, we have $m^{(2)}(g)=m^{(1)}\!\left( a_1 \mapsto h(a_1) \right)=m^{(1)}(h).$
A similar argument works when $g(a_1,a_2)=h(a_2)$ for all $(a_1,a_2) \in A^2$. Assume that the assertion holds for some $n \ge 1$. Let $g \in \Ca^1_b(A^{n+1}, \mathbb{R})$ be independent of the $i$-th coordinate. By definition,
	$$ m^{(n+1)}(g) =	m\!\left(a_{n+1} \mapsto m^{(n)}\big(g(-,a_{n+1})\big) \right). $$
	For a fixed $a_{n+1} \in A$, the map $(a_1,\dots,a_n) \mapsto	g(a_1,\dots,a_n,a_{n+1})$ is independent of the $i$-th coordinate (if $i\le n$) and $a_{n+1}$ (if $i=n+1$). By the induction hypothesis, $m^{(n)}\big(g(-,a_{n+1})\big)=m^{(n-1)}\big(h(-,a_{n+1})\big)$. Substituting this in the definition of $m^{(n+1)}$ yields
	$$m^{(n+1)}(g)	=	m\!\left(	a_{n+1} \mapsto	m^{(n)}\big(g(-,a_{n+1})\big)	\right)=	m\!\left(	a_{n+1} \mapsto	m^{(n-1)}\big(h(-,a_{n+1})\big)	\right)=m^{(n)}(h),$$
	which is desired.
\end{proof}

We recall that abelian groups are amenable \cite[Theorem 3.3]{MR3726870}, and hence admit a left-invariant mean. The following result is a quandle analogue of the Gromov Mapping Theorem for abelian quandle extensions. 

\begin{theorem}\label{gmt for quandles}
Let $X$ be a quandle, $A$ be an abelian group equipped with a left-invariant mean $m$, and $\varphi \colon X \times X \to A$ be a quandle $2$-cocycle. Let $E = X \times_\varphi A$ be the associated abelian extension with the quandle operation
	\begin{eqnarray*}
		(x,a) \star (y,b) = \bigl(x * y,\, a + \varphi(x,y)\big)
	\end{eqnarray*}
	and let $\pi \colon E \to X$ be the projection map $(x, a) \mapsto x$. For each $n \ge 1$, define a map $	\kappa_R^n \colon C_{R,b}^n(E, \mathbb{R}) \longrightarrow C_{R,b}^n(X, \mathbb{R})$ by
	\begin{eqnarray*}
			(\kappa_R^n f)(x_1,\dots,x_n)
			=
			m^{(n)}\big(\widetilde{f}_{x_1,\dots,x_n}\big),
		\end{eqnarray*}
		where $\widetilde{f}_{x_1,\dots,x_n} : A^n \longrightarrow \mathbb{R}$ is defined as 
		$$\widetilde{f}_{x_1,\dots,x_n}(a_1,\dots,a_n)=	f\big((x_1,a_1),\dots,(x_n,a_n)\big).$$
		Then the following assertions hold:
	\begin{enumerate}
		\item $\kappa_R^*$ is well-defined. 
		\item $\kappa_R^*$ is a cochain map.
		\item $\kappa_R^{*}  \pi_b^* = \mathrm{id}$.
	\end{enumerate}
	Consequently, $\pi_b^n:H_{Q,b}^n(X,\mathbb{R}) \hookrightarrow H_{Q,b}^n(E,\mathbb{R})$ is injective for all $n \ge 1$.
\end{theorem}

\begin{proof}
For assertion (1), let $f \in C_{R,b}^n(E,\mathbb{R})$ and fix $(x_1,\dots,x_n)\in X^n$. Then, we have 
		\begin{eqnarray*}
			\big|(\kappa_R^n f)(x_1,\dots,x_n)\big|
			&=&
			\big|m^{(n)}(\widetilde{f}_{x_1,\dots,x_n})\big|
			\\
			&\le&
			\|\widetilde{f}_{x_1,\dots,x_n}\|_\infty, \quad \text{by \eqref{inequality}}
			\\
			&=&
			\sup_{a_1,\dots,a_n \in A}
			\big|
			f\big((x_1,a_1),\dots,(x_n,a_n)\big)
			\big|
			\\
			&\le&
			\|f\|_\infty.
	\end{eqnarray*}
	Hence, $\|\kappa_R^n f\|_\infty \le \|f\|_\infty$, and the cochain $\kappa_R^n f$ is bounded, which makes the map $\kappa_R^n $ well-defined for each $n \ge 1$.
	
	\par
	For assertion (2), let $\delta_E^n$ and $\delta_X^n$ denote the coboundary operators for $E$ and $X$, respectively. Let $f \in C_{R,b}^n(E,\mathbb{R})$ and fix $(x_1,\dots,x_{n+1})\in X^{n+1}$. Define
		\begin{eqnarray*}
			\widetilde{\delta_E^n f}_{x_1,\dots,x_{n+1}}(a_1,\dots,a_{n+1})
			=
			(\delta_E^n f)\big((x_1,a_1),\dots,(x_{n+1},a_{n+1})\big).
		\end{eqnarray*}
		Since $\delta_E^n f$ is bounded, it follows that $\widetilde{\delta_E^n f}_{x_1,\dots,x_{n+1}} \in \Ca_b^1(A^{n+1}, \mathbb{R}),$
		and therefore
		\begin{eqnarray*}
			(\kappa_R^{n+1}\delta_E^n f)(x_1,\dots,x_{n+1})
			=
			m^{(n+1)}(\widetilde{\delta_E^n f}_{x_1,\dots,x_{n+1}}).
		\end{eqnarray*}
By expanding $\delta_E^n f$, we can write $\widetilde{\delta_E^n f}_{x_1,\dots,x_{n+1}}
		=
		\sum_{i=2}^{n+1} (-1)^i
		\big(
		G^{(1)}_i - G^{(2)}_i
		\big),$
		where
\begin{Small}
		\begin{eqnarray*}
			G^{(1)}_i(a_1,\dots,a_{n+1}) &=& f((x_1,a_1),\dots,\widehat{(x_i,a_i)},\dots,(x_{n+1},a_{n+1})) \quad \textrm{and}\\
			G^{(2)}_i(a_1,\dots,a_{n+1}) &=& f((x_1 * x_i,a_1+\varphi(x_1,x_i)),\dots,(x_{i-1} * x_i,a_{i-1}+\varphi(x_{i-1},x_i)), (x_{i+1},a_{i+1}), \dots,(x_{n+1},a_{n+1})).
		\end{eqnarray*}	
\end{Small}
Here, $G^{(1)}_i, G^{(2)}_i \in \Ca_b^1(A^{n+1}, \mathbb{R})$.	By linearity, we have
		\begin{eqnarray}\label{kappa {n+1}delta f}
			(\kappa_R^{n+1}\delta_E^n f)(x_1,\dots,x_{n+1})
			=
			\sum_{i=2}^{n+1} (-1)^i
			\big(
			m^{(n+1)}(G^{(1)}_i)
			-
			m^{(n+1)}(G^{(2)}_i)
			\big).
		\end{eqnarray}	
For each $2 \le i\le n+1$, let
		\begin{eqnarray*}
			v_i=
			\big(
			\varphi(x_1,x_i),\dots,\varphi(x_{i-1},x_i),0,\dots,0
			\big)
			\in A^{n+1}
		\end{eqnarray*}
and let $L_{v_i} : A^{n+1} \longrightarrow A^{n+1}$ be the translation by $v_i$. Then $G^{(2)}_i= H^{(2)}_i L_{v_i}$, where $H^{(2)}_i$ is the bounded map given by
		\begin{eqnarray*}
			H^{(2)}_i(a_1,\dots,a_{n+1})
			=
			f\big((x_1*x_i,a_1),\dots,(x_{i-1}*x_i,a_{i-1}),
			(x_{i+1},a_{i+1}),\dots,(x_{n+1},a_{n+1})\big).
		\end{eqnarray*}
		Since $m^{(n+1)}$ is left-invariant, we have $ m^{(n+1)}(G^{(2)}_i)= m^{(n+1)}(H^{(2)}_i).$ Now, by Lemma \ref{mean on product}, $m^{(n+1)}(G^{(1)}_i)=	m^{(n)}(G^{(1)}_i)$ and $m^{(n+1)}(G^{(2)}_i)=	m^{(n)}( G^{(2)}_i)=m^{(n)}( H^{(2)}_i)$. This, together with \eqref{kappa {n+1}delta f}, gives
		\begin{eqnarray*}
			(\kappa_R^{n+1}\delta_E^n f)(x_1,\dots,x_{n+1})
			&=&
			\sum_{i=2}^{n+1}
			(-1)^i
			\Big(
			m^{(n)}(\widetilde{f}_{x_1,\dots,\widehat{x_i},\dots,x_{n+1}})- m^{(n)}(\widetilde{f}_{x_1*x_i,\dots,x_{i-1}*x_i,x_{i+1},\dots,x_{n+1}})
			\Big)
			\\
			&=&
			(\delta_X^n \kappa_R^n f)(x_1,\dots,x_{n+1}).
	\end{eqnarray*}
		Hence, we have $\kappa_R^{n+1} \delta_E^n = \delta_X^n \kappa_R^n$ for each $n \ge 1$, and $\kappa_R^*$ is a cochain map. 
	\par
	
	Finally, for assertion (3), let $f \in C_{R,b}^n(X, \mathbb{R})$. Then $\pi^n_b f \in C_{R,b}^n(E, \mathbb{R})$ satisfies
	\begin{eqnarray*}
		(\pi^n_b f)((x_1,a_1),\dots,(x_n,a_n)) = f(x_1,\dots,x_n).
	\end{eqnarray*}
Thus, the associated map on $A^n$ is constant, that is, $\widetilde{\pi^n_b f}_{x_1,\dots,x_n}(a_1,\dots,a_n)=f(x_1,\dots,x_n)$ for all $(a_1, \ldots, a_n) \in A^n$. Consequently, we have 
\begin{eqnarray*}
(\kappa_R^n \pi^n_b f)(x_1,\dots,x_n) &= & m^{(n)}\big(\widetilde{\pi^n_b f}_{x_1,\dots,x_n}\big)\\
&= & m^{(n)}( f(x_1,\dots,x_n)\, \mathbf{1})\\
&= & f(x_1,\dots,x_n)m^{(n)}( \mathbf{1})\\
&=& f(x_1,\dots,x_n).
\end{eqnarray*}
This gives $\kappa_R^n \pi^n_b = \mathrm{id}$ for each $n \ge 1$, which is desired.
\par

It remains to pass from bounded rack cochains to normalized bounded quandle cochains. By the rack--quandle splitting~\cite[Theorem~8.2 and Example~8.5]{MR3558231}, for each quandle $Y$, there is a degree-wise bounded cochain projection
$$ P_Y^*:C_{R,b}^\ast(Y;\mathbb R) \longrightarrow C_{Q,b}^\ast(Y;\mathbb R)$$ such that $$ P_Y^*|_{C_{Q,b}^\ast(Y;\mathbb R)}=\operatorname{id} \qquad\text{and}\qquad P_Y^*\delta^*=\delta^* P_Y^*.$$
We define $$ \kappa^*_Q :=	P_X^*\, \kappa_R^*\big|_{C_{Q,b}^\ast(E;\mathbb R)}.$$
Then $\kappa_Q^*$ is a cochain map and $$\kappa_Q^*\pi_b^\ast =P_X^*\kappa_R^*\pi_b^\ast=P_X^*=\operatorname{id}$$ on $C_{Q,b}^\ast(X;\mathbb R)$. Consequently, the homomorphism $$\pi_b^n:H_{Q,b}^n(X;\mathbb R)\longrightarrow H_{Q,b}^n(E;\mathbb R)$$ is injective for each $n\geq1$.
\end{proof}

\begin{rem}
The map $\pi_b^*$ is not surjective in general. For example, let $X$ be a quandle with finitely many connected components such that $\Inn(X)$ is amenable and bounded. If $E = X \times \mathbb{R}$ is the trivial abelian extension, then $\Inn(E) \cong \Inn(X)$, and hence is amenable and bounded. It follows from \cite[Theorem 6.3 and Remark 6.4]{MR4779104} that
$$H_{Q,b}^n(X,\mathbb{R}) \cong \Map_b(\mathcal{O}(X)^n,\mathbb{R}) \quad \textrm{and} \quad H_{Q,b}^n(E,\mathbb{R}) \cong \Map_b(\mathcal{O}(E)^n,\mathbb{R}).
$$
Here, $\mathcal{O}(-)$	denotes the number of connected components of a quandle. Since $\mathcal{O}(E)\cong \mathcal{O}(X) \times \mathbb{R}$, the dimension of the space $\Map_b(\mathcal{O}(E)^n,\mathbb{R})$ is strictly larger than that of the space $\Map_b(\mathcal{O}(X)^n,\mathbb{R})$. Hence, the induced map $\pi_b^* \colon H_{Q,b}^n(X,\mathbb{R}) \longrightarrow H_{Q,b}^n(E,\mathbb{R})$ cannot be surjective.
\end{rem}
\medskip

\section{Idempotents in integral quandle algebras of Dehn quandles}\label{section Idempotents in integral quandle algebras}
Let $X$ be a quandle and $\mathbf{k}$ an associative ring with unity ${\bf 1}$. The {\it quandle ring/algebra} $\mathbf{k}[X]$ is defined as the free $\mathbf{k}$-module
$$
\mathbf{k}[X]=\left\{\sum_{i=1}^r n_ix_i~\Big| ~n_i\in \mathbf{k},~ x_i \in X  \text{ and } r \ge 1 \right\}
$$
 generated by $X$ and equipped with the ring multiplication given by
$$
\left(\sum_{i=1}^r n_i x_i\right)\left(\sum_{j=1}^s m_j y_j\right)=\sum_{i,j}^{r,s} n_i m_j \, x_i\ast y_j.
$$

These non-associative algebras were introduced in \cite{MR3977818} as an effort to incorporate ring- and module-theoretic techniques into the study of quandles.  An element $u=\textstyle\sum_{i=1}^r n_i x_i\in \mathbf{k}[X]$ with each $n_i \ne 0$ is said to have {\it length} $r$, and the set $\{x_1, x_2, \ldots, x_r \}$ is called the {\it support} of $u$. 
\par

An element $u \in \mathbf{k}[X]$ is called an {\it idempotent} if $u^2=u$. Since every element of $X$ can be viewed as an element of $\mathbf{k}[X]$ via the map $x \mapsto {\bf 1} x$, the first quandle axiom implies that elements of $X$ are idempotents in $\mathbf{k}[X]$. Such idempotents are referred to as {\it trivial idempotents} of $\mathbf{k}[X]$. Any idempotent of $\mathbf{k}[X]$ of length more than one is called a {\it non-trivial idempotent}. Characterizing the complete set of idempotents in integral quandle algebras of notable classes of quandles is a problem of significant interest. Recent works \cite{arXiv:2601.07057, MR4450681, MR4565221} have addressed this question for free quandles and certain finite quandles. In this section, we explore the corresponding problem for the Dehn quandle $\D_g$ of a closed orientable surface $S_g$.

\begin{lem}\label{GIC more than zero curves}
Let $g \ge 1$, and let $\alpha_1$ and $\alpha_2$ be distinct simple closed curves on $S_g$ satisfying $i(\alpha_1, \alpha_2) \ge 1$. Then $\{\alpha_1, \alpha_2, \alpha_1\ast \alpha_2, \alpha_2\ast \alpha_1 \}$ is a set of distinct curves on $S_g$.
\end{lem}

\begin{proof}
If $\alpha_1=\alpha_1\ast \alpha_2=T_{\alpha_2}(\alpha_1)$, then $i(\alpha_1,\alpha_2)=0$, a contradiction. If $\alpha_2=\alpha_1\ast \alpha_2$, then $\alpha_1\ast \alpha_2=\alpha_2 * \alpha_2$. This gives $\alpha_1=\alpha_2$, which is again a contradiction. Similarly, we get $\alpha_1 \ne \alpha_2\ast \alpha_1$ and $\alpha_2 \ne \alpha_2\ast \alpha_1$.  Finally, if $\alpha_1\ast \alpha_2=\alpha_2\ast \alpha_1$, then $T_{\alpha_2}(\alpha_1)=T_{\alpha_1}(\alpha_2)$. Using \eqref{conjugation by mcg}, this gives $T_{\alpha_2}T_{\alpha_1}T_{\alpha_2}^{-1}=T_{\alpha_1}T_{\alpha_2}T_{\alpha_1}^{-1}$. By \cite[Proposition 3.11-3.13]{MR2850125}, if $i(\alpha_1,\alpha_2)=1$, then $T_{\alpha_1}$ and $T_{\alpha_2}$ satisfy the braid relation. Further, by ~\cite[Theorem 3.14]{MR2850125}, if $i(\alpha_1,\alpha_2) \ge 2$, then $\langle T_{\alpha_1}, T_{\alpha_2} \rangle$ is a free group of rank two. Both cases lead to contradictions, and hence $\alpha_1\ast \alpha_2\neq \alpha_2\ast \alpha_1$.
\end{proof}

An {\it integral affine combination of curves} is an element $\textstyle\sum_{i=1}^r n_i\alpha_i$ of  $\Z[\D_g]$ such that $n_i\in\mathbb Z\setminus\{0\}$ and $\textstyle\sum_{i=1}^r n_i=1$.

\begin{prop}\label{idempotent structure}
For each $g \ge 1$, any integral affine combination of pairwise disjoint simple closed curves is an idempotent of $\Z[\D_g]$. Furthermore, any idempotent of $\Z[\D_g]$ of length at most three is an integral affine combination of pairwise disjoint simple closed curves.
\end{prop}

\begin{proof}
Consider the element $u=\textstyle\sum_{i=1}^r n_i\alpha_i\in \Z[\D_g]$, where $\textstyle\sum_{i=1}^r n_i=1$ and $\{\alpha_1,\alpha_2,\dots,\alpha_r\}$ is a multicurve on $S_g$. We see that 
$$u^2= \left(\sum_{i=1}^r n_i\alpha_i \right) \left(\sum_{j=1}^r n_j\alpha_j \right)=\sum_{i,j}^{r} n_i n_j \, \alpha_i\ast \alpha_j=\sum_{i,j}^{r} n_i n_j \, \alpha_i=u,$$ that is, $u$ is an idempotent in $\Z[\D_g]$. 
\par

Let $u \in \Z[\D_g]$ be an idempotent of length at most three. The assertion is immediate for idempotents of length one. Then we have the following two cases:
\par

\noindent {\bf Case 1.} Suppose that $u=n_1\alpha_1+n_2\alpha_2$ is of length two. Then $u^2=u$ gives 
\begin{equation}\label{idempotent case1}
(n_1^2-n_1)\alpha_1+(n_2^2-n_2)\alpha_2+n_1n_2(\alpha_1\ast \alpha_2+\alpha_2*\alpha_1)=0.
\end{equation} 
Assume that $i(\alpha_1,\alpha_2) \ge 1$. By Lemma \ref{GIC more than zero curves}, $\alpha_1$, $\alpha_2$, $\alpha_1\ast \alpha_2$, and $\alpha_2\ast \alpha_1$ are distinct curves. It now follows from \eqref{idempotent case1} that $n_1^2-n_1=n_2^2-n_2=n_1n_2=0$. Since $n_1,n_2\neq 0$, we arrive at a contradiction. Hence, we must have $i(\alpha_1,\alpha_2)=0$. In this case, $\alpha_1\ast \alpha_2=\alpha_1$, $\alpha_2\ast \alpha_1=\alpha_2$ and \eqref{idempotent case1} takes the form $n_1(n_1+n_2-1)\alpha_1+n_2(n_2+n_1-1)\alpha_2=0$. Since $n_1,n_2\neq 0$, we get $n_1+n_2=1$, which is desired.
\par

\noindent {\bf Case 2.} Suppose that $u=n_1\alpha_1+n_2\alpha_2+n_3\alpha_3$ is of length three. Then the idempotency of $u$ gives
\begin{tiny}
\begin{equation}\label{idempotent case2}
(n_1^2-n_1)\alpha_1+(n_2^2-n_2)\alpha_2+(n_3^2-n_3)\alpha_3+n_1n_2(\alpha_1\ast \alpha_2+\alpha_2\ast \alpha_1)+n_2n_3(\alpha_2\ast \alpha_3+\alpha_3\ast \alpha_2)+n_1n_3(\alpha_1\ast \alpha_3+\alpha_3\ast \alpha_1)=0
\end{equation}
\end{tiny}
\noindent {\bf Case 2a.} If $\{\alpha_1,\alpha_2,\alpha_3\}$ is a multicurve, then \eqref{idempotent case2} becomes
$$
n_1(n_1+n_2+n_3-1)\alpha_1+n_2(n_1+n_2+n_3-1)\alpha_2+n_3(n_1+n_2+n_3-1)\alpha_3=0.
$$
Since $\alpha_i$'s are distinct curves and $n_i$'s are non-zero, we obtain $n_1+n_2+n_3=1$, which is desired.

\noindent {\bf Case 2b.} Assume that $i(\alpha_1,\alpha_2)\ge 1$ and $i(\alpha_1,\alpha_3)=i(\alpha_2,\alpha_3)=0$. Then \eqref{idempotent case2} becomes
$$
(n_1^2-n_1+n_1n_3)\alpha_1+(n_2^2-n_2+n_2n_3)\alpha_2+(n_3^2-n_3+n_2n_3+n_1n_3)\alpha_3+n_1n_2(\alpha_1\ast \alpha_2+\alpha_2\ast \alpha_1)=0.
$$

It follows from Lemma \ref{GIC more than zero curves} that $\{\alpha_1, \alpha_2, \alpha_1\ast \alpha_2,\alpha_2\ast \alpha_1 \}$ is a set of distinct curves. If $\alpha_3=\alpha_1\ast \alpha_2=T_{\alpha_2}(\alpha_1)$, then it follows from \cite[Proposition 3.2]{MR2850125} that $0=i(\alpha_3,\alpha_1)=i(T_{\alpha_2}(\alpha_1),\alpha_1)=i(\alpha_1,\alpha_2)^2 \ge 1$, which is a contradiction. Similarly, $\alpha_3 \ne \alpha_2\ast \alpha_1$, and hence $\{\alpha_1, \alpha_2, \alpha_3, \alpha_1\ast \alpha_2, \alpha_2\ast \alpha_1 \}$ is a set of distinct curves. It now follows from \eqref{idempotent case2} that $n_1n_2=0$, which is a contradiction. Hence, this case is not possible. Similar arguments show that the following cases are not possible:
\begin{enumerate}[(i)]
\item $i(\alpha_1,\alpha_3) \ge 1$ and $i(\alpha_1,\alpha_2)=i(\alpha_2,\alpha_3)=0$.
\item $i(\alpha_2,\alpha_3) \ge 1$ and $i(\alpha_1,\alpha_3)=i(\alpha_1,\alpha_2)=0$.
\end{enumerate}

\noindent {\bf Case 2c.} Assume that $i(\alpha_1,\alpha_2) \ge 1$, $i(\alpha_1,\alpha_3) \ge 1$ and $i(\alpha_2,\alpha_3)=0$. Then \eqref{idempotent case2} becomes
\begin{small}
$$
(n_1^2-n_1)\alpha_1+(n_2^2-n_2+n_2n_3)\alpha_2+(n_3^2-n_3+n_2n_3)\alpha_3+n_1n_2(\alpha_1\ast \alpha_2+\alpha_2\ast \alpha_1)+n_1n_3(\alpha_1\ast \alpha_3+\alpha_3\ast \alpha_1)=0
$$
\end{small}
If $\alpha_3=\alpha_1\ast \alpha_2=T_{\alpha_2}(\alpha_1)$, then $0=i(\alpha_3,\alpha_2)=i(T_{\alpha_2}(\alpha_1),\alpha_2)=i(T_{\alpha_2}(\alpha_1),T_{\alpha_2}(\alpha_2))=i(\alpha_1,\alpha_2) \ge 1$, which is a contradiction. Similarly, if $\alpha_3=\alpha_2\ast \alpha_1=T_{\alpha_1}(\alpha_2)$, then $0=i(\alpha_3,\alpha_2)=i(T_{\alpha_1}(\alpha_2),\alpha_2)=i(\alpha_1,\alpha_2)^2 \ge 1$, which is a contradiction. If $\alpha_2=\alpha_1\ast \alpha_3=T_{\alpha_3}(\alpha_1)$, then $0=i(\alpha_2,\alpha_3)=i(T_{\alpha_3}(\alpha_1),\alpha_3)=i(T_{\alpha_3}(\alpha_1),T_{\alpha_3}(\alpha_3))=i(\alpha_1,\alpha_3) \ge 1$, a contradiction. Similarly, if $\alpha_2=\alpha_3\ast \alpha_1=T_{\alpha_1}(\alpha_3)$, then $0=i(\alpha_2,\alpha_3)=i(T_{\alpha_1}(\alpha_3),\alpha_3)=i(\alpha_1,\alpha_3)^2 \ge 1$, again a contradiction. It follows from Lemma \ref{GIC more than zero curves} that $\{\alpha_1, \alpha_2, \alpha_3, \alpha_1\ast\alpha_2, \alpha_2\ast \alpha_1,\alpha_1\ast\alpha_3, \alpha_3\ast \alpha_1 \}$ is a set of distinct curves. This implies that $n_1n_2=n_1n_3=0$. Since $n_i$'s are non-zero, we arrive at a contradiction. Similar arguments show that the following cases are not possible:
\begin{enumerate}[(i)]
\item $i(\alpha_1,\alpha_2)\ge 1$, $i(\alpha_2,\alpha_3) \ge 1$ and $i(\alpha_1,\alpha_3)=0$.
\item $i(\alpha_2,\alpha_3)\ge 1$, $i(\alpha_1,\alpha_3)\ge 1$ and $i(\alpha_1,\alpha_2)=0$.
\end{enumerate}

\noindent {\bf Case 2d.} Assume that $i(\alpha_1,\alpha_2)\ge 1$, $i(\alpha_1,\alpha_3)\ge 1$ and $i(\alpha_2,\alpha_3)\ge 1$. Suppose that $\alpha_3=\alpha_1\ast \alpha_2= T_{\alpha_2}(\alpha_1)$ and $i(\alpha_1,\alpha_2)=1$. It follows that $i(\alpha_1,\alpha_3)=i(\alpha_2,\alpha_3)=1$. Then the braid relation in $\M_g$ \cite[Proposition 3.12]{MR2850125} implies that $\alpha_2\ast \alpha_3=T_{\alpha_3}(\alpha_2)=T_{T_{\alpha_2}(\alpha_1)}(\alpha_2)=T_{\alpha_2}T_{\alpha_1}T_{\alpha_2}^{-1}(\alpha_2)=T_{\alpha_2}T_{\alpha_1}(\alpha_2)=\alpha_1$ and similarly $\alpha_3\ast \alpha_1=T_{\alpha_1}(\alpha_3)=T_{\alpha_1}T_{\alpha_2}(\alpha_1)=\alpha_2$. Thus, \eqref{idempotent case2} takes the form 
\begin{Small}
$$(n_1^2-n_1+n_2n_3)\alpha_1+(n_2^2-n_2+n_1n_3)\alpha_2+(n_3^2-n_3+n_1n_2)\alpha_3+ n_1n_2(\alpha_2\ast \alpha_1)+n_1n_3(\alpha_1\ast \alpha_3)+n_2n_3(\alpha_3\ast \alpha_2)=0.$$
\end{Small} 
By the change of coordinate principle~\cite[Section 1.3]{MR2850125}, it follows that $\{\alpha_1, \alpha_2, \alpha_3, \alpha_2\ast \alpha_1, \alpha_1\ast \alpha_3, \alpha_3\ast \alpha_2 \}$ is a set of distinct curves. See Figure~\ref{fig:six_distinct_curves_torus} showing these curves on the torus.
\begin{figure}
\centering
\includegraphics[scale=.8]{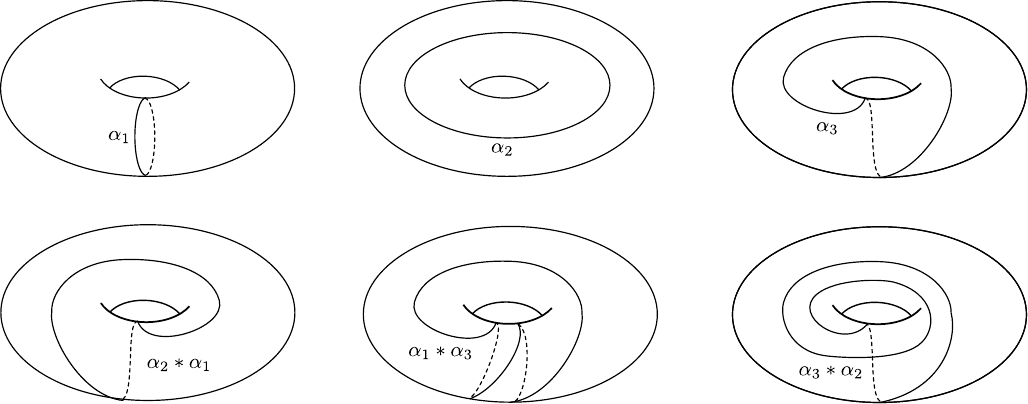}
\caption{Six distinct curves on the torus.}
\label{fig:six_distinct_curves_torus}
\end{figure}
Thus, the preceding equation gives $n_1n_2=n_2n_3=n_1n_3=0$. Since $n_i$'s are non-zero, we arrive at a contradiction.
\par
Next, suppose that $\alpha_3=\alpha_1\ast \alpha_2$ and $i(\alpha_1,\alpha_2) \ge 2$. In this case, $\langle T_{\alpha_1},T_{\alpha_2}\rangle$ is a free group of rank two. Since $\alpha_3=\alpha_1\ast \alpha_2$, it follows that $\{\alpha_1, \alpha_2, \alpha_3, \alpha_2\ast \alpha_1,\alpha_3\ast \alpha_1, \alpha_1\ast \alpha_3, \alpha_2\ast \alpha_3,\alpha_3\ast \alpha_2 \}$ is a set of distinct curves, else we will get a relation in the group $\langle T_{\alpha_1},T_{\alpha_2}\rangle$. Using this in \eqref{idempotent case2} leads to a contradiction. Hence, $\alpha_3\neq \alpha_1\ast \alpha_2$. 
\par 

Similarly, we can show that $\alpha_3\neq \alpha_2\ast \alpha_1$. In view of Lemma \ref{GIC more than zero curves}, we conclude that $\{\alpha_1,\alpha_2,\alpha_3,\alpha_1\ast\alpha_2,\alpha_2\ast \alpha_1\}$ is a set of distinct curves. Similar arguments show that $\{\alpha_1,\alpha_2,\alpha_3,\alpha_1\ast\alpha_3,\alpha_3\ast \alpha_1\}$ and $\{\alpha_1,\alpha_2,\alpha_3,\alpha_3\ast\alpha_2,\alpha_2\ast \alpha_3\}$ are sets of distinct curves. This implies that each $\alpha_i$ is distinct from the remaining eight curves. Thus, \eqref{idempotent case2} gives $n_1^2-n_1=n_2^2-n_2=n_3^2-n_3=0$. Since $n_i$'s are non-zero, we conclude that $n_1=n_2=n_3=1$. Consequently, \eqref{idempotent case2} becomes
$$\alpha_1\ast \alpha_2+\alpha_2\ast \alpha_1+\alpha_2\ast \alpha_3+\alpha_3\ast \alpha_2+\alpha_1\ast \alpha_3+\alpha_3\ast \alpha_1=0,
$$
which is not possible over integers. Hence, if an element of $\Z[\D_g]$ of length three is an idempotent, then it must be an integral affine combination  of pairwise disjoint simple closed curves. This completes the proof. 
\end{proof}

Next, we determine the complete set of idempotents in the quandle algebra  $\mathbb{Z}[\mathcal{D}_1]$. Note that any two distinct essential simple closed curves on the torus have positive geometric intersection number. We identify the elements of the Dehn quandle $\mathcal{D}_1$ with primitive vectors of $\mathbb Z \times \mathbb Z$ modulo $\pm 1$.  If $a, b$ are  primitive vectors in $\mathbb Z \times \mathbb Z$ representing  simple closed curves $\alpha, \beta  \in \mathcal{D}_1$ respectively, then $ i(\alpha,\beta)=|\det(a,b)|$. With this identification, the quandle operation is given by
$$
 \alpha*\beta=\alpha+\det( \alpha,\beta) \beta.
$$

\begin{lem}
\label{torus idempotent lemma 1}
The quandle algebra $\mathbb{Z}[\mathcal{D}_1]$ contains no non-trivial idempotent whose support curves have equal pairwise intersection numbers.
\end{lem}

\begin{proof}
Let $u\in \mathbb{Z}[\mathcal{D}_1]$ be an idempotent whose support curves have equal pairwise intersection numbers. Suppose that the length of $u$ is $r \ge 4$. Let $\alpha_1, \alpha_2, \alpha_3, \alpha_4$ be four distinct curves in the support of $u$. Then we have $|\det(\alpha_i,\alpha_j)|=n \ge 1$ for all $i \ne j$.  After applying an element of $SL(2,\mathbb Z)$, we write $\alpha_1=(1,0)$. Since $|\det(\alpha_1,\alpha_2)|=n$, we choose the sign of the primitive vector such that $\alpha_2=(a,n)$,  where $\gcd(a,n)=1$. Since $|\det(\alpha_1,\alpha_3)|=n$, we can similarly choose the primitive vector such that $\alpha_3=(b,n)$,  where $\gcd(b,n)=1$. This gives 
$$n = |\det(\alpha_2,\alpha_3)|=|\det((a,n),(b,n))| =|a n - b n| = n|a-b|.$$ Hence $ |a-b|=1$, and therefore $\alpha_3=(a+1,n)$ or $(a-1,n)$. Since $|\det(\alpha_1,\alpha_4)|=n$, we write $\alpha_4=(c,n)$. This yields
$$n = |\det(\alpha_2,\alpha_4)|=|\det((a,n),(c,n))| =|a n - c n| = n|a-c|.$$
 Hence $|a-c|=1$, and therefore $ \alpha_4=(a+1,n)$ or $(a-1,n)$. Since $\alpha_3\neq\alpha_4$, they must be the opposite choices. Without loss of generality, suppose that
$ \alpha_3=(a+1,n)$  and $ \alpha_4=(a-1,n)$. Then we have
$$n= |\det(\alpha_3,\alpha_4)| =| \det ((a+1, n), (a-1,n))| =|(a+1)n-(a-1)n|=2n,$$
which is a contradiction. Hence, we must have $r\le 3$. In this case, it follows from Proposition 6.2 that $u$ is a trivial idempotent. This proves the assertion.
\end{proof}

\begin{lem}\label{torus idempotent lemma 2}
Let $u\in\mathbb Z[\mathcal{D}_1]$ be an idempotent. If two distinct ordered pairs of curves from the support of $u$ attain the maximum intersection number, then their products are distinct.
\end{lem}

\begin{proof}
Let $u\in\mathbb Z[\mathcal{D}_1]$ be a non-trivial idempotent, and let $$M= \max \{|\det(\alpha, \beta)| \mid ~\textrm{$\alpha$ and $\beta$ are distinct primitive vectors in the support of $u$} \}.$$ 
If $M=1$, then all ordered pairs of primitive vectors from the support of $u$ have intersection number 1. By Lemma \ref{torus idempotent lemma 1}, the idempotent $u$ is trivial, which is a contradiction. Hence we must have $M>1$. Let $(\alpha,\beta)$ and $(\gamma,\delta)$ be ordered pairs of primitive vectors from the support of $u$ such that
$$|\det(\alpha,\beta)|=|\det(\gamma,\delta)|=M.$$
Suppose that $\zeta:=\alpha*\beta=\gamma*\delta$.  Choose orientations of $\alpha,\beta$ such that $\det(\alpha,\beta)=M$. Then, we have $\zeta=\alpha+M\beta$. Similarly, choose orientations of $\gamma,\delta$ such that $\det(\gamma,\delta)=M$.  Then, we have  $\zeta=\gamma+M\delta$. Consequently, we obtain
$$
 \det(\alpha,\zeta)=\det(\alpha,\alpha+M\beta)=M^2~\textrm{and}~  \det(\gamma,\zeta)=\det(\gamma,\gamma+M\delta)=M^2.
$$
Hence $ \det(\gamma-\alpha,\zeta)=0$, and therefore  $\gamma-\alpha=k\zeta$ for some $k\in\mathbb Z$. If $k\neq0$, then we have
$$
|\det(\alpha,\gamma)|=|\det(\alpha,\alpha+k\zeta)|=|k|\det(\alpha,\zeta)=|k|M^2>M,
$$
which contradicts the maximality of $M$.  Hence, we must have $k=0$. In this case, $\gamma=\alpha$, and it follows from $\alpha+M\beta=\gamma+M\delta$ that $\delta=\beta$.  Thus, the ordered pairs of primitive vectors coincide. 
\end{proof}

\begin{theorem}\label{idempotents in dehn quandle of torus}
\label{torus idempotents theorem}
The quandle algebra $\mathbb Z[\mathcal{D}_1]$ has no non-trivial idempotents.
\end{theorem}

\begin{proof}
Let $u\in\mathbb Z[\mathcal{D}_1]$ be a non-zero idempotent. Assume that $u$ has length at least two, and let
$$M= \max \{|\det(\alpha, \beta)| \mid ~\textrm{$\alpha$ and $\beta$ are distinct primitive vectors in the support of $u$} \}.$$ 
Suppose that $M>1$. Let $\alpha, \beta$ be primitive vectors from the support of $u$ such that $M=\det(\alpha, \beta)$. Set $\zeta=\alpha*\beta=\alpha+M\beta$. Then
$$
\det(\alpha,\zeta)
=\det(\alpha,\alpha+M\beta)
=M^2.
$$
Thus $|\det(\alpha,\zeta)|=M^2>M.$ Since $M$ is the maximum intersection number among the curves in the support of $u$, it follows that $\zeta$ does not belong to the support of $u$. On the other hand, the coefficient of $\zeta$ in $u^2$ coming from the ordered pair $(\alpha,\beta)$ is a nonzero integer.  Since $u^2=u$, the total coefficient of $\zeta$ in $u^2$ must be zero.  Hence, there exists at least one more ordered pair $(\gamma, \delta)\neq (\alpha,\beta)$ in the support of $u$ such that $\zeta=\gamma*\delta$, so that cancellation can occur. After orienting $\gamma$ and $\delta$, if necessary, let $\det(\gamma,\delta)=N>0$. Suppose that $N<M$. Since $\zeta=\gamma+N \delta$, we obtain $\alpha+M \beta=\gamma+N\delta$, which further gives $\gamma=\alpha+M \beta-N\delta$. Putting $ k=\det(\delta,\beta)$, we have
$$ \det(\gamma,\beta) =\det(\alpha,\beta)+M\det(\beta,\beta)-N\det(\delta,\beta) =M-Nk. $$
On the other hand, we have

$$ N=\det(\gamma,\delta) =\det(\alpha,\delta)+M\det(\beta,\delta)=\det(\alpha,\delta)-Mk, $$
and hence $ \det(\alpha,\delta)=N+Mk$. If $k=0$, then $\beta$ and $\delta$ are parallel primitive vectors, and hence $\delta=\pm \beta$. But then
$$N=|\det(\gamma, \delta)|=|\det(\gamma+N\delta, \delta)|=|\det(\zeta,\delta)| =|\det(\zeta,\beta)| =|\det(\alpha+M\beta,\beta)| =M,$$
which contradicts $N<M$. If $k>0$,  then $$|\det(\alpha,\delta)|=N+Mk>M,$$ a contradiction to the maximality of $M$. If $k<0$, then 
$$ |\det(\gamma, \beta)|=M-Nk =M+N|k| >M, $$
again a contradiction to the maximality of $M$. Hence, $N<M$ is impossible, and  $\det(\gamma,\delta)=M$. Thus, the two ordered pairs $(\alpha,\beta)$ and $(\gamma,\delta)$ of primitive vectors have the maximum intersection number $M$ and give the same product $\zeta$. By Lemma \ref{torus idempotent lemma 2}, we have $(\alpha,\beta)=(\gamma,\delta)$, which contradicts the assumption that they are distinct. Hence, we must have $M=1$. In this case, all the curves from the support of $u$ have intersection number 1. The theorem now follows from Lemma \ref{torus idempotent lemma 1}.
\end{proof}

It is known that the quandle $\mathcal{D}_1$ is isomorphic to the fundamental quandle of the trefoil knot~\cite[Theorem 3.1]{MR2583322}. Thus, we have the following result.

\begin{cor}\label{trefoil has no non-trivial idempotents}
The integral quandle algebra of the fundamental quandle of the trefoil knot has no non-trivial idempotents.
\end{cor}

Proposition \ref{idempotent structure} and Theorem \ref{idempotents in dehn quandle of torus} suggest the following conjecture.

\begin{conj}
For each $g \ge 2$, the non-zero idempotents of $\Z[\D_g]$ are precisely the integral affine combinations of pairwise disjoint simple closed curves.
\end{conj}

The following related problem is also worth exploring.

\begin{prob}
Determine the complete set of idempotents in the integral quandle algebra $\mathbb{Z}[X]$, where $X\in \{\mathcal{C}_g,\mathcal{W}_g,\mathcal{W}_g^{1},\mathcal{W}_g^{+} \}$,
for every $g\geq 1$.
\end{prob}

\begin{rem}
The fundamental quandle of an unlink with $k \ge 1$ components is the free quandle of rank $k$. Consequently, \cite[Theorem 5.3 and Corollary 5.4]{MR4565221} imply that the integral quandle algebra of the fundamental quandle of an unlink has no non-trivial idempotents. In contrast, the fundamental quandle of the Hopf link is the trivial quandle on two elements. By \cite[Proposition 4.1]{MR4450681}, every idempotent in its integral quandle algebra is an integral affine combination of the generators, and hence it admits non-trivial idempotents. 
\end{rem}

The preceding remark, together with Corollary \ref{trefoil has no non-trivial idempotents}, motivates the following problem.

\begin{prob}
Determine the complete set of idempotents in the integral quandle algebra $\mathbb{Z}[Q(L)]$, where $Q(L)$ is the fundamental quandle of a link $L$ in the 3-sphere.
\end{prob}

Let $\widehat{\C_g}$ be the set of free homotopy classes of oriented closed curves on $S_g$, and let $\mathbb{Z}\widehat{\C_g}$ be the free abelian group generated by $\widehat{\C_g}$. For curves $\alpha,\beta\in \widehat{\C_g}$ intersecting transversely,  define $[-,-]:\mathbb{Z}\widehat{\C_g} \times \mathbb{Z}\widehat{\C_g}\to \mathbb{Z}\widehat{\C_g}$ by
$$
[\alpha,\beta]=\sum_{p\in \alpha \cap \beta} i(p) \, \alpha \ast_{p} \beta,
$$
where $\alpha \ast_{p} \beta$ is the loop product of $\alpha$ and $\beta$ at the intersection point $p$ and $i(p)$ is the sign of the intersection at $p$. In the landmark work \cite[Theorem 5.3]{MR0846929},  Goldman proved that $\mathbb{Z}\widehat{\C_g}$ is a Lie algebra with Lie bracket $[-,-]$. We conclude with the following problem.

\begin{prob}
For each $g \ge 1$, does there exist a meaningful relationship between the Goldman Lie algebra $\mathbb{Z}\widehat{\C_g}$ and the integral quandle algebra $\mathbb{Z}[\C_g]$?
\end{prob}

\smallskip
\textbf{Acknowledgements.} 
PK acknowledges the Institute Postdoctoral Fellowship from IISER Mohali. DS acknowledges the Institute Postdoctoral Fellowship from IISER Tirupati. MS acknowledges support from the ARG-MATRICS grant with Grant No. ANRF/ARGM\\/2025/000406/MTR.

\end{document}